\documentclass[a4paper,10pt]{amsart}
\usepackage[colorlinks,linkcolor=blue,citecolor=blue]{hyperref}
\usepackage{latexsym, amssymb, amsmath, amsthm, bbm}
\usepackage[all]{xy}

\usepackage{anysize}\marginsize{30mm}{30mm}{30mm}{30mm}
\def \To{\longrightarrow}
\def \dim{\operatorname{dim}}

\def \Hom{\operatorname{Hom}}

\def \Ind{\operatorname{Ind}}

\def \ord{\operatorname{order}}
\def \Rep{\operatorname{Rep}}
\def \Aut{\operatorname{Aut}}
\def \id{\operatorname{Id}}
\def \ker{\operatorname{Ker}}

\def \ord{\textsf{o}}

\def \M{\mathrm{M}}

\def \Z{\mathbb{Z}}

\numberwithin{equation}{section}

\newtheorem{theorem}{Theorem}[section]
\newtheorem{lemma}[theorem]{Lemma}
\newtheorem{proposition}[theorem]{Proposition}
\newtheorem{corollary}[theorem]{Corollary}
\newtheorem{definition}[theorem]{Definition}

\newtheorem{remark}[theorem]{Remark}

\begin{document}
\title{Graded elementary quasi-Hopf algebras of tame representation type$^\dag$}\thanks{$^\dag$Supported by PCSIRT IRT1264, SRFDP 20130131110001, NSFC 11371186 and SDNSF ZR2013AM022.}

\subjclass[2010]{16T05, 18D10, 16G60} \keywords{quasi-Hopf algebra, tensor category, tame representation type}

\author{Hua-Lin Huang}
\address{School of Mathematics, Shandong University, Jinan 250100, China} \email{hualin@sdu.edu.cn}

\author{Gongxiang Liu}
\address{Department of Mathematics, Nanjing University, Nanjing 210093, China} \email{gxliu@nju.edu.cn}

\author{Yu Ye}
\address{School of Mathematics, University of Science and Technology of China, Hefei 230026, China} \email{yeyu@ustc.edu.cn}
\date{}
\maketitle

\begin{abstract}
The class of graded elementary quasi-Hopf algebras of tame type is classified. Combining with our previous work \cite{QHA3}, this completes the trichotomy for such class of algebras according to their representation types. In addition, new examples of genuine elementary quasi-Hopf algebras, and accordingly finite pointed tensor categories, are provided.
\end{abstract}

\section{Introduction}
Tensor categories appear, as a ubiquitous algebraic structure,  in many areas of mathematics and theoretical physics, for example representation theory, quantum algebra, topology, quantum computation, conformal field theory, and topological orders. The classification problem of tensor categories has been a central research theme for the last several decades. 

In the theory of tensor categories, the idea of Tannakian formalism is indispensable which aims to concretize abstract tensor categories as module categories of concrete algebras. Though a criterion for the reconstruction of general tensor categories is not yet available, Etingof and Ostrik observed in \cite{EO} the important fact that any finite tensor category whose objects all have integer Frobenius-Perron dimension is equivalent to the module category of a finite dimensional quasi-Hopf algebra. This reduces the classification problem of finite pointed tensor category to that of elementary quasi-Hopf algebras. Recall that an algebra is said to be elementary, if it is finite dimensional and its simple modules are 1-dimensional.  

The classification problem of finite pointed tensor categories motivates many new constructions of elementary quasi-Hopf algebras. Certainly this is closely related to that of finite dimensional pointed Hopf algebras, as the dual of the latter are elementary Hopf algebras. Therefore, the beautiful theory of finite dimensional pointed Hopf algebras of Andruskiewitch and Schneider (see \cite{AS} and references therein) serves as the starting point for the investigation of elementary quasi-Hopf algebras. Etingof and Gelaki started the pioneering work and published a series of papers \cite{EG3,G,EG,EG4}, in which they provided a new method of constructing genuine elementary quasi-Hopf algebras from known finite dimensional pointed Hopf algebras and obtained an explicit classification for genuine elementary quasi-Hopf algebras over cyclic groups of prime order. Along the same line, in \cite{A} Angiono extended Etingof and Gelaki's construction and achieved a complete classification of genuine elementary quasi-Hopf algebras over cyclic groups whose order is not divisible by $2, 3, 5, 7.$ Here by ``genuine" is meant the quasi-Hopf algebra is not twist equivalent to a Hopf algebra.

On the other hand, the well-developed ideas and techniques from the representation theory of finite dimensional algebras (see e.g. \cite{ARS}) can be naturally applied to the classification problem of finite pointed tensor categories. In particular, the handy quiver techniques help to visualize the constructions of elementary quasi-Hopf algebras and their representations, and thus finite pointed tensor categories. This paper aims to contribute more classification results within the quiver framework of tensor categories and quasi-Hopf algebras initiated in \cite{QHA1, QHA2, QHA3}. The ultimate goal is to provide a complete classification of elementary quasi-Hopf algebras and the associated finite pointed tensor categories. In the representation theory of finite dimensional algebras, the concept of representation types is a valuable invariant which measures the complexity of representation categories. The well known trichotomy theorem of Drozd \cite{Dr} asserts that the representation category of an algebra is either of finite, tame, or wild type and these three are exclusive. Roughly speaking, the cardinality of the set of finite dimensional indecomposable representations is either finite, infinite but at any fixed dimension almost all contained in a finite number of one-parameter families,  or infinite but not as the previous case. Representation types provide a natural standard for the classification of finite dimensional algebras and their representations.

Unfortunately, it is generally believed to be an impossible mission to give an explicit trichotomy for all finite dimensional algebras via their representation types. However, such an aim seems reasonable for elementary quasi-Hopf algebras over an algebraically closed field of characteristic $0.$ For elementary Hopf algebras, an explicit trichotomy have been achieved in \cite{LL,L2,HL,L3}. For elementary quasi-Hopf algebras, those of finite representation type have been obtained in \cite{QHA3}. As a continuation of \cite{QHA3}, the purpose of this paper is to provide a complete classification of elementary graded quasi-Hopf algebras of tame type. This will complete the trichotomy of the class of elementary graded quasi-Hopf algebras and the corresponding class of finite pointed tensor categories in the sense of Drozd.   

Comparing with the case of Hopf algebras, or equivalently finite pointed tensor categories with fiber functors, we are facing two obvious difficulties. The first  is that the classification procedure developed in \cite{HL,L3} is not applicable directly to the quasi-Hopf case. A key step in the Hopf case is to decompose a graded elementary Hopf algebra $H$ as the biproduct or bosonization $R_{H}\# H/J_{H}$ where $R_{H}$ is a local subalgebra and thus Ringel's remarkable classification result \cite{R} about local algebras can be applied. Though in the quasi-Hopf case one may still define biproduct or bosonization accordingly (see e.g. \cite{BE}), the resulting algebra $R_{H}$ is not, in general, a usual associative algebra and if we make it into an associative algebra artificially then $R_{H}$ is not a subalgebra. The second is that the associators of quasi-Hopf algebras are generally nontrivial and we need to deal with $3$-cocycles over abelian groups which are not cyclic. To overcome the first difficulty, our basic idea is to realize a quasi-Hopf algebra of tame type as a subalgebra of a Hopf algebra with the same representation type. It turns out that the recently developed methods of equivariantization and de-equivariantization \cite{EG,G,DGNO} can be applied. A key observation is that representation type is an invariant under the equivariantization and de-equivariantization procedures. This helps to construct new graded quasi-Hopf algebras of tame type from known Hopf algebras. For the second, we may use the unified formulae of normalized $3$-cocycles of finite abelian groups obtained in \cite{HLY}. Another key observation is that we only need to deal with $3$-cocycles $\omega$ on the direct product of two cyclic groups $\mathbbm{Z}_{\mathbbm{m}}\times \mathbbm{Z}_{\mathbbm{n}}$ and they are resolvable in bigger finite abelian groups, i.e., there exist  group epimorphisms $\pi:\;\mathbbm{Z}_{\mathbbm{m}^2}\times \mathbbm{Z}_{\mathbbm{n}^2}\to \mathbbm{Z}_{\mathbbm{m}}\times \mathbbm{Z}_{\mathbbm{n}}$ such that the pull-back $\pi^{\ast}(\omega)$ are coboundaries.  This observation allows us to prove that the constructed graded quasi-Hopf algebras previously exhaust all genuine graded quasi-Hopf algebras of tame type.

The paper is organized as follows. In Section 2, some preliminaries are provided. In particular, the basic ingredients of quiver methods and the definitions of equivariantization and de-equivariantization are recalled. In Section 3, we show that representation type is an invariant under the equivariantization and de-equivariantization procedures. As a technical preparation, Section 4 is devoted to the analysis of the generators of abelian groups. In Section 5, some new quasi-Hopf algebras are constructed. The main result is formulated in Section 6, which states that any tame graded elementary quasi-Hopf algebra is twist equivalent to either a Hopf algebra as given in \cite{HL} or a quasi-Hopf algebra as constructed in Section 5. 
     
Throughout of this paper, we work over an algebraically closed field $k$ of characteristic zero. For convenience, we fix some notations. Given any natural numbers $m,n$, let $[\frac{m}{n}]$ denote the floor function of $m,n$, i.e., the biggest integer which is not bigger than $\frac{m}{n}$.  Let $G$ be a finite group and $g \in G$, by $\ord(g)$ we denote the order of $g$.

\section{Preliminaries}

In this section, we will recall the definition of quasi-Hopf algebras, equivariantization and de-equivariantization appeared in \cite{DGNO}, representation types and some basic facts about Hopf quivers \cite{CR} (or, equivalently covering quivers \cite{GS}).

\subsection{Quasi-Hopf algebras.} A quasi-bialgebra $(H,\M, \mu, \Delta, \varepsilon, \phi)$ is a
$k$-algebra $(H,\M,\mu)$ with two algebra morphisms $\Delta:\;H\to
H\otimes H$ (the comultiplication) and $\varepsilon:\; H\to
k$ (the counit), and an invertible element $\phi\in H\otimes
H\otimes H$ (called the associator), such that
\begin{gather}
(\id\otimes \Delta)\Delta(a)\phi=\phi(\Delta\otimes \id)\Delta(a),\;\; \forall a\in H,\\
(\id\otimes \id\otimes\Delta)(\phi)(\Delta\otimes \id\otimes \id)(\phi)=(1\otimes \phi)
(\id\otimes \Delta\otimes \id)(\phi)(\phi\otimes 1),\\
(\varepsilon\otimes \id)\Delta=\id=(\id\otimes \varepsilon)\Delta,\\
(\id\otimes \varepsilon\otimes \id)(\phi)=1\otimes 1.
\end{gather}

We denote $\phi=\sum X^{i}\otimes Y^{i}\otimes Z^{i}$ and
$\phi^{-1}=\sum \overline{X}^{i}\otimes \overline{Y}^{i}\otimes
\overline{Z}^{i}$. Then a quasi-bialgebra $H$ is called a
quasi-Hopf algebra if there is a linear algebra antimorphism
$S:\;H\to H$ (the antipode) and two elements
$\alpha,\beta\in H$ satisfying for all $a\in H$,
\begin{gather}
\sum S(a_{(1)})\alpha a_{(2)}=\alpha \varepsilon(a),\;\;\sum a_{(1)}\beta S(a_{(2)})=\beta \varepsilon(a),\\
\sum X^{i}\beta S(Y^{i})\alpha Z^{i}=1=\sum S(\overline{X^{i}})\alpha \overline{Y^{i}}\beta S(\overline{Z^{i}}).
\end{gather}

We say that an invertible element $J\in H\otimes H$ is a \emph{twist} of
$H$ if it satisfies $(\varepsilon\otimes \id)(J)=(\id\otimes
\varepsilon)(J)=1$. For a twist $J=\sum f_{i}\otimes g_{i}$ with
inverse $J^{-1}=\sum \overline{f_{i}}\otimes \overline{g_{i}}$, set
$$\alpha_{J}:= \sum S(\overline{f_{i}})\alpha \overline{g_{i}},\;\;
\beta_{J}:= \sum f_{i}\beta S(g_{i}).$$ It is well known that given a
twist $J$ of $H$ then one can
construct a new quasi-Hopf algebra structure \cite{D}
$H^{J}=(H,\Delta_{J},\varepsilon,\Phi_{J},S,\alpha_{J},\beta_{J})$
on the algebra $H$, where
$$\Delta_{J}(a)=J\Delta(a)J^{-1},\;\; \forall a\in H,$$
and
$$\Phi_{J}=(1\otimes J)(\id\otimes \Delta)(J)(\Delta\otimes \id)(J^{-1})(J\otimes 1)^{-1}.$$

\begin{definition} Two quasi-Hopf algebras $H_{1}$ and $H_{2}$ are called twist equivalent if there is a twist $J$ of 
$H_{1}$ such that $H_{2} \cong H_{1}^{J}$ as quasi-Hopf algebras.
\end{definition} 

\subsection{Equivariantization and  de-equivariantization}
Let $\mathcal{C}$ be a $k$-linear category. Let $\underline{End}(\mathcal{C})$ denote the category of $k$-linear functors $\mathcal{C}\to
\mathcal{C}$. This is a monoidal $k$-linear category (the tensor product is the composition of functors).

For a finite group $G$, let $\underline{G}$ denote the corresponding monoidal category: the objects of $\underline{G}$ are elements of $G$,
the only morphisms are the identities and the tensor product is given by multiplication in $G$.

\begin{definition} An action of $G$ on $\mathcal{C}$ is a monoidal functor $F:\;\underline{G}\to \underline{End}(\mathcal{C})$.
\end{definition}

\begin{remark} If $\mathcal{C}$ is a tensor category over $k$, then we use $\underline{Aut}(\mathcal{C})$ to denote the category whose objects are tensor auto-equivalences of $\mathcal{C}$ and whose morphisms are isomorphisms of tensor functors. And an action of $G$ on $\mathcal{C}$ is defined correspondingly as a monoidal functor $F:\; \underline{G}\to \underline{Aut}(\mathcal{C})$.
\end{remark}

Let $G$ be a finite group acting on a $k$-linear abelian category $C$. For any $g\in G$ let $F_{g}\in \underline{End}(\mathcal{C})$ be the corresponding functor and for any $g,h\in G$ let $\gamma_{g,h}$ be the isomorphism $F_{g}\circ F_{h}\simeq F_{gh}$ that defines the tensor structure on the functor $F:\;\underline{G}\to \underline{End}(\mathcal{C})$. A $G$-\emph{equivariant object} of $\mathcal{C}$ is an object $X\in \mathcal{C}$ together with isomorphisms $u_{g}: \; F_{g}(X)\simeq X$ such that the diagram 

\begin{figure}[hbt]
\begin{picture}(100,60)(0,0)
 \put(-40,50){\makebox(0,0){$F_{g}(F_{h}(X)) $}}
 \put(-5,50){\vector(1,0){70}}
 \put(85,50){\makebox(0,0){$F_{g}(X) $}}

 \put(-40,0){\makebox(0,0){$F_{gh}(X) $}}
 \put(-5,0){\vector(1,0){70}}
 \put(85,0){\makebox(0,0){$X$}}

  \put(-30,40){\vector(0,-1){30}}
   \put(85,40){\vector(0,-1){30}}

 \put(30,60){\makebox(0,0){$F_{g}(u_{h}) $}}
  \put(30,10){\makebox(0,0){$u_{gh}$}}

   \put(-15,25){\makebox(0,0){$\gamma_{g,h}$}}
     \put(95,25){\makebox(0,0){$u_{g}$}}
\end{picture}
\end{figure}
\noindent commutes for all $g,h\in G.$

 One defines morphisms of equivariant objects to be morphisms in $\mathcal{C}$ commuting with $u_{g},\;g\in G$. The category of $G$-equivariant objects of $\mathcal{C}$ will be denoted by $\mathcal{C}^{G}$, which is called the \emph{equivariantization} of $\mathcal{C}$.

Let $A:=Fun(G)$ be the algebra of functions $G\to k$. The group $G$ acts on $A$ by left translations, so $A$ can be considered as an algebra in the monoidal category $\Rep(G)$. Let $\mathcal{D}$ be an $k$-abelian category with a $k$-linear action of $\Rep(G)$, that is, there is a $k$-linear monoidal functor $F:\;\Rep(G)\to \underline{End}(\mathcal{D})$. The category of $A$-modules in $\mathcal{D}$ will be called the \emph{de-equivariantization} of  $\mathcal{D}$, which will be denoted by $\mathcal{D}_{G}$.

Let $\mathcal{C}$ be a $k$-abelian category with a $G$-action $F:\; \underline{G}\to \underline{End}(\mathcal{C})$. Clearly, we have the following two functors:
$$ \textrm{The forgetful functor}\;\;\;\;\;\;\Phi:\; \mathcal{C}^{G}\to \mathcal{C}\;\;\;\;\;\;\;\;\;\;\;\;\;\;\;\;\;\;\;\;\;\;\;\;\;\;\;\;\;\;\;\;\;\;\;\;\;\;\;\;\;$$
$$\textrm{The induced functor}\;\;\;\;\;\Ind:\;\mathcal{C}\to \mathcal{C}^{G},\;\;\Ind(X):=\oplus_{g\in G}F_{g}(X).$$
The following lemma is derived from \cite[Lemma 4.6]{DGNO} directly (though \cite[Lemma 4.6]{DGNO} is proved in semisimple case, it is clearly true for non-semisimple case).
\begin{lemma}\label{l2.1}
For $X\in \mathcal{C}$ and $Y\in \mathcal{C}^{G}$, we have
$$X|\Phi({\Ind}(X)),\;\;\;\;\;\;\;\;Y|{\Ind}(\Phi(Y)).$$
Here and below, by $X | Z$ it is meant $X$ is a direct summand of $Z.$
\end{lemma}

The main property of equivariantization and de-equivariantization is recalled in the following
lemma, see \cite[Theorem 4.9]{DGNO} (as stated in \cite[Theorem 2.3]{EG2}, \cite[Theorem 4.9]{DGNO} is also true in non-semisimple case).
\begin{lemma}\label{l2.2} The procedures of equivariantization and de-equivariantization are inverse to each other. In particular, 
\begin{itemize}
\item[(i)] If there is a $G$-action on $\mathcal{C}$, then
$\mathcal{C}$ is equivalent to $(\mathcal{C}^{G})_{G}$.
\item[(ii)] If there is a $\Rep(G)$-action on $\mathcal{C}$, then
$\mathcal{C}$ is equivalent to $(\mathcal{C}_{G})^{G}$.
\end{itemize}
\end{lemma}

\subsection{Representation types.}   A finite dimensional algebra $A$ is said to be of \emph{finite representation type} provided there are finitely many iso-classes of indecomposable $A$-modules.
$A$ is of \emph{tame  type} or $A$ is a \emph{tame} algebra if $A$
is not of finite representation type, whereas for any dimension
$d>0$, there are finite number of $A$-$k[T]$-bimodules $M_{i}$
which are free of finite rank as right $k[T]$-modules such that all but a finite
number of indecomposable $A$-modules of dimension $d$ are
isomorphic to $M_{i}\otimes_{k[T]}k[T]/(T-\lambda)$ for
$\lambda\in k$. We say that $A$ is of \emph{wild type} or $A$ is a
\emph{wild} algebra  if there is a finitely generated
$A$-$k\langle X,Y\rangle $-bimodule $B$ which is free as a right $k\langle X,Y\rangle $-module
such that the functor $B\otimes_{k\langle X,Y\rangle }-\;\;$ from Rep($k\langle X,Y\rangle $),
the category of finitely generated $k\langle X,Y\rangle $-modules, to Rep($A$),
the category of finitely generated $A$-modules, preserves
indecomposability and reflects isomorphisms. 

The well known Drozd's trichotomy theorem \cite{Dr} can be stated as follows.
\begin{theorem} \label{tt} Every finite dimensional algebra is either of finite, tame, or wild type and these three are mutually exclusive. 
\end{theorem}

\subsection{Quiver setting for quasi-Hopf algebras.}

A \emph{quiver} is an oriented graph $\Gamma=(\Gamma_{0},\;\Gamma_{1})$,
   where $\Gamma_{0}$ denotes the set of vertices and $\Gamma_{1}$
   denotes the set of arrows. Let $k\Gamma$ denote the associated \emph{path
   algebra} of the quiver $\Gamma.$ An ideal $I$ of $k\Gamma$
 is called \emph{admissible} if $J^{N}\subset I \subset J^{2}$ for
 some $N\geq 2$, where $J$ is the ideal generated by all arrows.

 For an elementary algebra $A$, by the Gabriel's Theorem,
 there is a unique quiver $\Gamma_{A}$,
 and an admissible ideal $I$ of $k\Gamma_{A}$, such that $A\cong
 k\Gamma_{A}/I$ (see \cite{ARS}). The quiver
 $\Gamma_{A}$ is called the \emph{Gabriel quiver} or \emph{Ext-quiver} of $A$.

 Next, let us recall the definition of \emph{covering quivers} (see
\cite{GS}). Let $G$ be a finite group and let
$W=(w_{1},w_{2},\ldots,w_{n})$ be a sequence of elements of $G$. We
say $W$ is a \emph{weight sequence} if, for each $g\in G$, the
sequences $W$ and $(gw_{1}g^{-1},gw_{2}g^{-1},\ldots,gw_{n}g^{-1})$
are the same up to a permutation. In particular, $W$ is closed under
conjugation. Define a quiver, denoted by $\Gamma_{G}(W)$, as
follows. The vertices of $\Gamma_{G}(W)$ is the set $\{v_{g}\}_{g\in
G}$ and the arrows are given by
$$\{(a_{i},g):\;v_{g^{-1}}\rightarrow v_{w_{i}g^{-1}}\;|\; i=1,2,\ldots,n, g\in G\}.$$
 We call this quiver the covering quiver (with
respect to $G$ and $W$).

The concept of covering quivers is indeed dual to the concept of Hopf quivers appeared in \cite{CR}.
Paralleling to the proof of \cite[Theorem 3.1]{QHA1}, we have

\begin{lemma}\label{l2.3} Let $H$ be an elementary quasi-Hopf algebra, then its Gabriel's quiver is a covering quiver.
\end{lemma}

Let $H$ be an elementary quasi-Hopf algebra, $Q(H)$ its Gabriel's quiver. By Lemma \ref{l2.3}, $Q(H)=\Gamma_{G}(W)$ for some group $G$ and some weight sequence $W.$ The following definition was given in \cite{L}.

\begin{definition} Let $H$ be an elementary quasi-Hopf algebra and $\Gamma_{G}(W)$ its Gabriel's quiver.  Define $n_{H}$ to be the
cardinal number of $W$ and call it the \emph{representation type number} of $H$.
\end{definition}

The following is a direct consequence of \cite[Theorem 2.1]{L}.

\begin{lemma}\label{l2.4} Let $H$ be an elementary quasi-Hopf algebra of tame type and $n_{H}$ its representation type number. Then we have
\begin{itemize}
\item[(i)] $H$ is of finite representation type if and only if $n_{H}=1$.
\item[(ii)] If $H$ is tame, then $n_{H}=2$.
\end{itemize}
\end{lemma}

\section{Invariance of representation type}
For convenience, we extend the notion of representation types of finite dimensional algebras to that for general finite abelian categories. The point is to get an invariant to measure the cardinality of the indecomposable objects of any such category. Recall that an abelian category $\mathcal{C}$ is called \emph{finite} if it is Morita equivalent to $\Rep(A)$ for some finite dimensional algebra $A$.  Note that this is equivalent to saying
\begin{itemize}
\item[(i)] $\mathcal{C}$ is a $k$-linear abelian category and has finite-dimensional spaces of morphisms;
\item[(ii)] every object of $\mathcal{C}$ has finite length;
\item[(iii)] there are enough projective objects in  $\mathcal{C}$; and
\item[(iv)] there are finitely many isomorphism classes of simple objects.
\end{itemize}

\begin{definition} Let $\mathcal{C}$ be a finite abelian category and assume that $\mathcal{C}$ is Morita equivalent to $\Rep(A)$. The category $\mathcal{C}$ is said to be of finite type, tame type, or wild type if the algebra $A$ is of finite, tame, or wild representation type respectively.
\end{definition}

Clearly, the representation type of an algebra is an invariant of Morita equivalence, hence the above definition is independent of the choice of $A$.

The main aim of this section is to show that the type of a finite abelian category is preserved under the procedures of equivariantization or de-equivariantization. 

 \begin{proposition}\label{p3.2} $\mathcal{C}$ is a finite abelian category if and only if $\mathcal{C}^{G}$ is so.
 \end{proposition}
 \begin{proof}  Clearly,  the properties (i), (ii), (iv) hold for $\mathcal{C}$ if and only if they hold for $\mathcal{C}^{G}$. It remains to prove this for (iii). The proof is divided into four claims.\\[1.5mm]
 \emph{Claim 1:  If $P$ is projective in $\mathcal{C}^{G}$, then $\Phi (P)$ is projective in $\mathcal{C}$.  } \\[1.5mm]
 \emph{Proof. } Assume we have the following diagram in $\mathcal{C}$:
\begin{figure}[hbt]
\begin{picture}(60,60)(0,0)
\put(0,60){\makebox(0,0){$ \Phi (P)$}}
\put(-20,0){\makebox(0,0){$M$}}
 \put(0,0){\vector(1,0){50}}
\put(60,0){\makebox(0,0){$N$}}
\put(15,45){\vector(1,-1){35}}
\put(20,10){\makebox(0,0){$\pi$}}
\put(30,40){\makebox(0,0){$f$}}
\end{picture}
\end{figure}

\noindent where $\pi$ is an epimorphism.  From this, we get a diagram in $\mathcal{C}^{G}$:\\
 
 \begin{figure}[hbt]
\begin{picture}(60,60)(0,0)
\put(0,60){\makebox(0,0){$ P$}}
\put(-20,0){\makebox(0,0){$\textrm{Ind}(M)$}}
 \put(0,0){\vector(1,0){50}}
\put(80,0){\makebox(0,0){$\textrm{Ind}(N)$}}
\put(15,45){\vector(1,-1){35}}
\put(20,10){\makebox(0,0){$\textrm{Ind}(\pi)$}}
\put(30,40){\makebox(0,0){$\tilde{f}$}}
\end{picture}
\end{figure}

\noindent where $\tilde{f}$ is gotten by  averaging $f$ with respect to $G$ but without multiplying $\frac{1}{|G|}.$ Since $P$ is projective in $\mathcal{C}^{G}$,  there is a $g\in \Hom_{\mathcal{C}^G}(P, \textrm{Ind}(M))$ such that $\tilde{f}=\textrm{Ind}(\pi)\circ g$. Note that $M|\textrm{Ind}(M)$ and $N|\textrm{Ind}(N)$. By restricting $\textrm{Ind}(\pi)$ to $M$, we have
$f=\pi\circ g$. Therefore, $\Phi( P)$ is projective in $\mathcal{C}$. \\[1.5mm]
 \emph{Claim 2:  If $P$ is projective in $\mathcal{C}$, then so is $\emph{Ind} (P)$ in $\mathcal{C}^{G}$.  } \\[1.5mm]
 \emph{Proof. } Assume we have the following diagram in $\mathcal{C}^{G}$:
\begin{figure}[hbt]
\begin{picture}(60,60)(0,0)
\put(0,60){\makebox(0,0){$ \textrm{Ind} (P)$}}
\put(-20,0){\makebox(0,0){$M$}}
 \put(0,0){\vector(1,0){50}}
\put(60,0){\makebox(0,0){$N$}}
\put(15,45){\vector(1,-1){35}}
\put(20,10){\makebox(0,0){$\pi$}}
\put(30,40){\makebox(0,0){$f$}}
\end{picture}
\end{figure}

\noindent where $\pi$ is an epimorphism.  Since $P$ is projective in $\mathcal{C}$, there is a $g\in \Hom_{\mathcal{C}}(\Phi(\textrm{Ind}( P)), \Phi(M))$ such that $f=\pi\circ g$.  Similarly, by averaging $g$ with respect to $G$ one gets $\tilde{g}$ and it is straightforward to show that $f=\pi\circ \tilde{g}.$ This implies that $\textrm{Ind}( P)$ is projective in $\mathcal{C}^{G}$.\\[1.5mm]
\emph{Claim 3:  If (iii) holds in $\mathcal{C}^{G}$, then it holds in $\mathcal{C}$ too.}\\[1.5mm]
\emph{Proof. } Take any object $X\in \mathcal{C}$. By assumption, there is a projective object $P\in \mathcal{C}^{G}$ such that there is an epimorphism:
$$P\twoheadrightarrow \textrm{Ind}(X).$$
Since $X|\Phi( \textrm{Ind}(X))$, we have
$$\Phi( P)\twoheadrightarrow \Phi( \textrm{Ind}(X))\twoheadrightarrow X.$$
Therefore, Claim 1 implies $\mathcal{C}$ also has enough projective objects..\\[1.5mm]
\emph{Claim 4:  If (iii) holds in $\mathcal{C}$, then it holds in $\mathcal{C}^{G}$ as well.}\\[1.5mm]
\emph{Proof. } Take any object $X\in \mathcal{C}^{G}$. By assumption, there is a projective object $P\in \mathcal{C}$ such that there is an epimorphism:
$$P\twoheadrightarrow \Phi(X).$$
This implies
$$\textrm{Ind}( P)\twoheadrightarrow  \textrm{Ind}(\Phi(X))\twoheadrightarrow X.$$
Therefore, by Claim 2 $\mathcal{C}^{G}$ also has enough projective objects.
 
 \end{proof}

 \begin{corollary}\label{c3.3}
 $\mathcal{C}$ is a finite abelian category if and only if $\mathcal{C}_{G}$ is so.
 \end{corollary}
 \begin{proof} By Proposition \ref{p3.2},  $\mathcal{C}_{G}$ is a finite abelian category if and only if $(\mathcal{C}_{G})^{G}$ is so. Since $\mathcal{C}$ is always equivalent to $(\mathcal{C}_{G})^{G}$  by Lemma \ref{l2.2}, then the desired result follows.
 \end{proof}

Now we are ready to give the main result of the section. The proofs of the following two propositions are the same as those of Propositions  4.2, 4.4 in \cite{L2} except some minor technical points. For the sake of completeness, we include them here.

 \begin{proposition}\label{p3.4} $\mathcal{C}$ is of finite type if and only if $\mathcal{C}^{G}$ is so.
 \end{proposition}
 \begin{proof} ``\emph{Only if part}":   Let $\{X_1,\ldots,X_n\}$ be a complete set of non-isomorphic indecomposable objects in
 $\mathcal{C}$. Suppose $X$ is an indecomposable object in $\mathcal{C}^{G}$. Then $\Phi(X)=\bigoplus_{j=1}^{n}n_{j}X_j$
and so $\Ind(\Phi(X))=\bigoplus_{j=1}^{n}n_j\Ind(X_j)$. By Lemma \ref{l2.1}, $X|\Ind(\Phi(X))$ and so $X$ is a direct summand of
$\Ind(X_j)$ for some $j$. Therefore, the non-isomorphic indecomposable $\mathcal{C}^{G}$-summands of all the $\Ind(X_j)$ give a complete set of non-isomorphic indecomposable objects in $\mathcal{C}^{G}$. Obviously this set is  finite, thus $\mathcal{C}^{G}$ is of finite type.

``\emph{If part}":  The proof is almost identical to the preceding part.  Let $\{Y_1,\ldots,Y_n\}$ be a complete set of non-isomorphic indecomposable objects in
 $\mathcal{C}^{G}$. Suppose $Y$ is an indecomposable object in $\mathcal{C}$. Then $\Ind(Y)=\bigoplus_{j=1}^{n}n_{j}Y_j$
and so $\Phi(\Ind(Y))=\bigoplus_{j=1}^{n}n_j\Phi(Y_j)$. By Lemma \ref{l2.1}, $Y|\Phi(\Ind(Y))$ and so $Y$ is a direct summand of
$\Phi(Y_j)$ for some $j$. Therefore, the non-isomorphic indecomposable $\mathcal{C}$-summands of all the $\Phi(Y_j)$ give a complete set of non-isomorphic indecomposable objects in $\mathcal{C}$. As the set is finite, $\mathcal{C}$ is of finite type.
 \end{proof}

 Next we consider the case of tame type. Before moving on, we need to recall a technical lemma. Let $\Lambda$ be an arbitrary finite dimensional algebra. In \cite{L2}  the category $GC(\Lambda),$ called \emph{generic category},  was defined to investigate indecomposable $\Lambda$-modules. By definition, its objects are $\Lambda$-$k[T]$-bimodules which are finitely generated free as right $k[T]$-modules and morphisms are $\Lambda$-$k[T]$-morphisms. 
 
 \begin{lemma}\cite[Lemma 4.3]{L2}\label{l4.2} Let $X\in GC(\Lambda)$. Then $X$ is indecomposable in $GC(\Lambda)$ if and only if $X\otimes_{k[T]}k[T]/(T-\lambda) $ is indecomposable as a $\Lambda$-$k[T]/(T-\lambda)$-bimodule for some $\lambda\in k$.
 \end{lemma}

 \begin{proposition}\label{p3.5} $\mathcal{C}$ is of tame type if and only if $\mathcal{C}^{G}$ is so.
 \end{proposition}
 \begin{proof} ``\emph{Only if part}":   Suppose that $\mathcal{C}=\Rep(A)$ (for some finite dimensional algebra $A$) is tame and we will prove that $\mathcal{C}^{G}=\Rep(B)$ (for some finite dimensional algebra $B$) is also tame. Clearly, $\mathcal{C}^{G}$ is not of finite type since otherwise $\mathcal{C}$ is of finite type too by Proposition \ref{p3.4}. Let $d$ be a positive integer and $X$ an indecomposable $B$-module. Assume $\dim_{k}X=d$ and $\Phi(X)=X_1\oplus \cdots \oplus X_m$ as a decomposition into indecomposable $A$-modules. Thus, $X|\Ind(X_i)$ for some $i\in \{1,\ldots,m\}$ by Lemma \ref{l2.1}.
 
Since $A$ is a tame algebra, there are finitely many $A$-$k[T]$-bimodules $M_j\;(j=1,\ldots,n)$ which are free with finite rank as right $k[T]$-modules such that almost all indecomposable $A$-modules of dimension $\le d$ are of the form $M_j\otimes_{k[T]}k[T]/(T-\lambda)$ for some $j$ and some $\lambda\in k$. In order to show that $\mathcal{C}^{G}$ is tame, there is no harm in assuming that $X_i\cong M_{i_j}\otimes_{k[T]}k[T]/(T-\lambda)$ for some $i_j\in \{1,\ldots,n\}$ and some $\lambda\in k$.  Clearly, we can still define $\Ind(M_{i_j})$ just as usual and $\Ind(M_{i_j})\in GC(B)$.

Assume $\Ind(M_{i_j})=\bigoplus_{k\in I_j}M_{i_j}^{k}$ for a finite index set $I_j,$ where each $M_{i_j}^{k}$ is indecomposable in $GC(B).$  By Lemma \ref{l4.2},  $M_{i_j}^{k}\otimes_{k[T]}k[T]/(T-\lambda)$ is indecomposable as an $B$-$k[T]/(T-\lambda)$-bimodule too. This is equivalent to saying that $M_{i_j}^{k}\otimes_{k[T]}k[T]/(T-\lambda)$ is indecomposable as a $B$-module since $k[T]/(T-\lambda)\cong k$. Since $$X|\Ind(X_i)=\Ind(M_{i_j})\otimes_{k[T]}k[T]/(T-\lambda)=\bigoplus_{k\in I_j}M_{i_j}^{k}\otimes_{k[T]}k[T]/(T-\lambda),$$ it follows that
 $X\cong M_{i_j}^{k}\otimes_{k[T]}k[T]/(T-\lambda)$ for some $i_j\in \{1,\ldots,n\}, \ k\in I_j$ and $\lambda\in k$. Since the set $\{M_{i_j}^{k}|1\leqslant i_j\leqslant n, k\in I_j\}$ is finite, $B$ and thus $\mathcal{C}^{G}$ is tame.

``\emph{If part}":  The proof can be carried out in a similar manner as that of ``\emph{If part}" 
and so omitted.
 \end{proof}

 \begin{theorem}\label{t3.6} Let $\mathcal{C}$ be a finite abelian category acted by a finite group $G.$ Then $\mathcal{C}$ and $\mathcal{C}^{G}$ are of the same type.
 \end{theorem}
 \begin{proof} Direct consequence of Propositions \ref{p3.4}, \ref{p3.5} and Drozd's trichotomy Theorem \ref{tt}.
 \end{proof}

 \begin{corollary}\label{c3.7} Let $\mathcal{C}$ be a finite abelian category with a $\Rep(G)$-action and $|G|<\infty$. Then $\mathcal{C}$ and $\mathcal{C}_{G}$ are of the same type.
 \end{corollary}
 \begin{proof} Direct consequence of Theorem \ref{t3.6} and Lemma \ref{l2.2}. \end{proof}

\section{Generators of abelian groups}
In this section, we consider the following very elementary question, which is important for our later computations: Given two generators $g,h$ of $\mathbbm{Z}_{m}\times \mathbbm{Z}_{n}=\langle g_1,g_2|g_{1}^{m}=g_{2}^{n}=1, g_{1}g_{2}=g_{2}g_{1}\rangle$ with $m|n$,
we know that there are integers $a,b,c,d$ such that $g=g_{1}^{a}g_{2}^{b},h=g_{1}^{c}g_{2}^{d}$ and $g,h$ generate
$\mathbbm{Z}_{m}\times \mathbbm{Z}_{n}$. The question is that can we simplify the expression of $g,h$? That is, up to an automorphism of $\mathbbm{Z}_{m}\times \mathbbm{Z}_{n}$, deduce the integers $a,b,c,d$ as simple as possible.

To this end, we call two generators $h_{1},h_{2}$ of $\mathbbm{Z}_{m}\times \mathbbm{Z}_{n}$ are \emph{standard} if there
is an automorphism $\sigma\in \Aut(\mathbbm{Z}_{m}\times \mathbbm{Z}_{n})$ satisfying $\sigma(g_1)=h_1, \sigma(g_2)=h_2$. The main result of this section can be formulated in the following form.

\begin{proposition}\label{p1} Assume that $g$ and $h$ generate the abelian group $\mathbbm{Z}_{m}\times \mathbbm{Z}_{n}$ with $m|n$, then there are integers $m_{1},m_{2},n_{1},n_{2},a,b$ such that

\emph{(i)} $m=m_{1}n_{1},\;n=m_{2}n_{2},\;\;\;\;m_{1}|m_{2},\; n_{1}|n_{2},\;\;\;\;(m_{2},n_{2})=1$;

\emph{(ii)} $0\leq a< n_{2},\;0\leq b<m_{2}$ and $$g=g_{2}h_{1}h_{2}^{a},\;\;h=g_{1}g_{2}^{b}h_{2}$$
where $g_1,g_{2}$ (resp. $h_{1},h_{2}$) are standard generators of $\mathbbm{Z}_{m_1}\times \mathbbm{Z}_{m_2}$ (resp. $\mathbbm{Z}_{n_1}\times \mathbbm{Z}_{n_2}$).
\end{proposition}

To prove the proposition, we need a preliminary lemma.

\begin{lemma}\label{l1} Let $p$ be a prime and $g,h$ be two generators of $\mathbbm{Z}_{p^i}\times \mathbbm{Z}_{p^j}$ with $i\leq j$. Assume that
the order of $g$ is not bigger than that of $h$. Then there exists $0\leq a< p^{j}$ such that
$$g=g_{1}g_{2}^{a},\;\;h=g_{2}$$
where $g_1,g_{2}$ are standard generators of $\mathbbm{Z}_{p^i}\times \mathbbm{Z}_{p^j}$.
\end{lemma}
\begin{proof} We start with two claims.\\[1.5mm]
\emph{Claim 1: The order of $h$ is $p^j$.}\\[1.5mm]
\emph{Proof of Claim 1}: Otherwise, the orders of $g$ and $h$ will be strictly smaller than $p^{j}$. Therefore, there is an $l<j$ such that
$g^{p^{l}}=h^{p^l}=1$. From this, we know the order of every element generated by $g,h$ is at most $p^{l}$. This is impossible since $g,h$ generate $\mathbbm{Z}_{p^i}\times \mathbbm{Z}_{p^j}$.\\[1.5mm]
\emph{Claim 2: $h$ is standard.}\\[1.5mm]
\emph{Proof of Claim 2}: Regarding $\mathbbm{Z}_{p^i}\times \mathbbm{Z}_{p^j}$ as a $\mathbbm{Z}_{p^j}$-module, by Claim 1 $\langle h\rangle =\mathbbm{Z}_{p^j}$ is
a submodule. It is well known that $\mathbbm{Z}_{p^j}$ is a QF-ring (QF means quasi-Frobenius). So there is a $\mathbbm{Z}_{p^j}$-module
$M$ such that $\mathbbm{Z}_{p^i}\times \mathbbm{Z}_{p^j}=\mathbbm{Z}_{p^i}\oplus \mathbbm{Z}_{p^j}=M\oplus  \langle h\rangle $. By the Structure Theorem of
finitely generated abelian groups, $M\cong \mathbbm{Z}_{p^i}$. Therefore, $h$ is standard.

Now we are in the position to give the proof of this lemma.  By Claim 2, there are standard generators $h_{1},h_{2}$ of
 $\mathbbm{Z}_{p^i}\times \mathbbm{Z}_{p^j}$ such that  $h=h_{2}$. Therefore, there are $s,t$ such that $g=h_{1}^{s}h_{2}^{t}$.
 Since $g,h$ generate $\mathbbm{Z}_{p^i}\times \mathbbm{Z}_{p^j}$, $h_1=g^{b}h^{c}$ for some $b,c$. So we have
 $$h_{1}=g^{b}h^{c}=h_{1}^{sb}h_{2}^{tb+c}.$$
 Thus $sb\equiv 1 (p^{i})$ and so $(b,p^{i})=1$. From this, define the automorphism
 $$\varphi:\;\mathbbm{Z}_{p^i}\times \mathbbm{Z}_{p^j}\rightarrow \mathbbm{Z}_{p^i}\times \mathbbm{Z}_{p^j},\;h_{1}\mapsto h_{1}^{b},\;h_2\mapsto h_{2}.$$
 Now let $g_{1}:=\varphi(h_1),\;g_{2}:=\varphi(h_2)$, we get the desired result.
\end{proof}

According to the proof of the previous lemma, we have the following 

\begin{corollary}\label{c2} Let $g,h$ be two generators of $\mathbbm{Z}_{m}\times \mathbbm{Z}_{n}$ with $m|n$. Assume the order of $h$ is $n$, then there are
standard generators $g_{1},g_{2}$ of $\mathbbm{Z}_{m}\times \mathbbm{Z}_{n}$ such that
$$g=g_{1}g_{2}^{a},\;\;h=g_{2}$$
for some $0\leq a<m$.
\end{corollary}

\emph{The proof of Proposition \ref{p1}.} Let $m=\prod_{i=1}^{t}p_{i}^{a_{i}},\;n=\prod_{i=1}^{t}p_{i}^{b_{i}}$
be the prime decomposition of $m,n$. By assumption, $a_{i}\leq b_{i}$. So $\mathbbm{Z}_{m}\times \mathbbm{Z}_{n}=\bigoplus_{i=1}^{t}(\mathbbm{Z}_{p_{i}^{a_i}}\times \mathbbm{Z}_{p_{i}^{b_i}})$.
For each $1\leq i\leq t$, let
$g_{i},h_{i}$ be standard generators of $\mathbbm{Z}_{p_{i}^{a_i}}\times \mathbbm{Z}_{p_{i}^{b_i}}$. So there are $s_{i},t_{i},x_{i},y_{i}$ such that
$$g=\prod_{i=1}^{t}g_{i}^{s_{i}}h_{i}^{t_{i}},\;\;h=\prod_{i=1}^{t}g_{i}^{x_{i}}h_{i}^{y_{i}}.$$
We call $g_{i}^{s_{i}}h_{i}^{t_{i}}$ (resp. $g_{i}^{x_{i}}h_{i}^{y_{i}}$) the $i$-th part of $g$ (resp. $i$-th part of $h$).  By Lemma \ref{l1}, the $i$-th part of $g$ or $h$
is standard. So we may assume without loss of generality that there exists $1\leq s\leq t$ such that for each $1\leq j\leq s$ the $j$-th part of $g$ is standard while the $l$-th part
of $h$ is standard for $s\leq l\leq t$.

Define $m_1:=\prod_{i=1}^{s}p_{i}^{a_{i}},\;m_{2}:=\prod_{i=1}^{s}p_{i}^{b_{i}},\;n_1:=\prod_{i=s+1}^{t}p_{i}^{a_{i}},\;n_{2}:=\prod_{i=s+1}^{t}p_{i}^{b_{i}}.$ and
$g_{1}':=\prod_{i=1}^{s}g_{i},\;g_{2}':=\prod_{i=1}^{s}h_{i},\;h_{1}':=\prod_{i=s+1}^{t}g_{i},\;h_{2}':=\prod_{i=s+1}^{t}h_{i}.$ According to our choice, one can assume that
the product of the first $s$ parts of $g$ equals to $g_{2}'$ and the product of the last $t-s$ parts of $h$ is $h_{2}'$. Using Corollary \ref{c2}, there are $0\leq a< n_{2}$ and
$0\leq b< m_{2}$ such that  $$g={g}_{2}{h}_{1}{h}_{2}^{a},\;\;h={g}_{1}{g}_{2}^{b}{h}_{2}$$
where ${g}_1,{g}_{2}$ (resp. ${h}_{1},{h}_{2}$) are standard generators of $\mathbbm{Z}_{m_1}\times \mathbbm{Z}_{m_2}$ (resp. $\mathbbm{Z}_{n_1}\times \mathbbm{Z}_{n_2}$).\qed

\section{Construction of quasi-Hopf algebras}
Gelaki invented a method of constructing new quasi-Hopf algebras from known Hopf algebras  in \cite{G}. Angiono generalized this method to classify elementary quasi-Hopf algebras over cyclic groups with minor condition \cite{A}. In this section, by applying Gelaki's method we construct graded elementary quasi-Hopf algebras of tame type which are over non-cyclic abelian group in general.

\subsection{Tame graded elementary Hopf algebras.}
Consider the following Hopf algebras.
 Let
$W=\mathbbm{Z}_{m}\times \mathbbm{Z}_{n}$ with $m$ even and $m|n$. Assume $g,h$ are two generators of $W$. Let $q,p$ be $\ord(g)$-th and $\ord(h)$-th
primitive roots of unity respectively. Take two integers
$l_{1},l_{2}$ with $l_{1}|m, l_{2}|n$ and set
$q_{2}:=q^{l_{1}},\;p_{1}:=p^{l_{2}}$. Assume that $p_{1}q_{2}$ is
an $l$-th primitive root of unity. Then the elementary Hopf algebra $H(m,n,l_{1},l_{2},g,h)$ is defined to be an associative algebra
generated by elements $x,y$ and $g,h$, with relations
$$g\;\textrm{and}\;h \;\textrm{generate}\; \mathbb{Z}_{m}\times \mathbb{Z}_{n},\;\;\;x^{2}=y^{2}=(xy)^{l}+(-q_{2})^{l}(yx)^{l}=0,$$
$$gxg^{-1}=q^{-1}x,\;gyg^{-1}=y,\;hxh^{-1}=x,\;hyh^{-1}=p^{-1}y.$$
The comultiplication $\Delta$, counit $\varepsilon$, and antipode
$S$ are given by
$$\Delta(g)=g\otimes g,\;\;\Delta(h)=h\otimes h$$
$$\Delta(x)=x\otimes 1+ g^{\frac{\ord(g)}{2}}h^{l_{2}}\otimes x,
\;\;\Delta(y)=y\otimes 1+ g^{l_{1}}h^{\frac{\ord(h)}{2}}\otimes y,$$
$$\varepsilon(g)=\varepsilon(h)=1,\;\;\;\;\varepsilon(x)=\varepsilon(y)=0$$
$$S(g)=g^{-1},\;\;S(h)=h^{-1},\;\;S(x)=-g^{\frac{\ord(g)}{2}}h^{-l_{2}}x,\;\;
S(y)=-g^{-l_{1}}h^{\frac{\ord(h)}{2}}y.$$


The main result of \cite{HL} can be stated as follows.
\begin{lemma}\label{l5.1} \cite[Theorems 4.9 and 4.16]{HL} Let $H$ be a
connected radically graded tame elementary Hopf algebra over an
algebraically closed field $k$ with characteristic $0$, then as a
Hopf algebra it is isomorphic to $H(m,n,l_{1},l_{2},g,h)$
for some $m,n,l_{1},l_{2}$ and two generators $g,h$ of $\mathbbm{Z}_{m}\times \mathbbm{Z}_{n}$.
\end{lemma}

Here ``connected" means that $H$ is connected as an algebra, see \cite{ARS}..

\subsection{3-cocycles over finite abelian groups.}
For any finite abelian group $G$, we obtained a complete set of representatives of $3$-cocycles (with coefficients in $k^{\ast}$) in \cite{HLY}, which is one of the key technical ingredients of this paper. Let's recall it.

Let $G=\mathbbm{Z}_{m_{1}}\times \cdots \times \mathbbm{Z}_{m_{n}}$ be a finite abelian group. A tuple of standard generators of $G$ is denoted by
$g_{1},\ldots,g_{n}$. For any natural number, $\zeta_{l}$ is an $l$-th primitive root of unity. As usual, we use $(B_{\bullet},\partial_{\bullet})$
to denote the bar resolution of $G$. Define $A$ to be the set of all sequences
\begin{equation*}(a_{1},\ldots,a_{l},\ldots,a_{n},a_{12},\ldots,a_{ij},\ldots,a_{n-1,n},a_{123},
\ldots,a_{rst},\ldots,a_{n-2,n-1,n})\;\;\;\;(\diamond)\end{equation*}
such that $ 0\leq a_{l}<m_{l},0\leq a_{ij}<(m_{i},m_{j}),0\leq a_{rst}<(m_{r},m_{s},m_{t})$ for $1\leq l\leq n,1\leq i<j\leq n,1\leq r<s<t\leq n$ where $a_{ij}$ and $a_{rst}$ are ordered by the lexicographic order. The sequence ($\diamond$) is denoted by $\underline{\mathbf{a}}$ for short.

For any $\underline{\mathbf{a}}\in A$, define a $\mathbb{Z}G$-module map
\begin{eqnarray}&&\omega_{\underline{\mathbf{a}}}:\;B_{3}\To k^{\ast}\\\notag
&&[g_{1}^{i_{1}}\cdots g_{n}^{i_{n}},g_{1}^{j_{1}}\cdots g_{n}^{j_{n}},g_{1}^{k_{1}}\cdots g_{n}^{k_{n}}]
\\\notag
&&\mapsto \prod_{l=1}^{n}\zeta_{l}^{a_{l}i_{l}[\frac{j_{l}+k_{l}}{m_{l}}]}
\prod_{1\leq s<t\leq n}\zeta_{(m_{s},m_{t})}^{a_{st}i_{t}[\frac{j_{s}+k_{s}}{m_{s}}]}
\prod_{1\leq r<s<t\leq n}\zeta_{(m_{r},m_{s},m_{t})}^{-a_{rst}k_{r}j_{s}i_{t}}.
\end{eqnarray}

\begin{lemma}\label{l5.2} \cite[Proposition 3.1]{HLY}  In the cochain complex $(B_{\bullet}^{\ast}, \partial_{\bullet}^{\ast})$, the set
$$\{\omega_{\underline{\mathbf{a}}}|\underline{\mathbf{a}}\in A\}$$
is a complete set of representatives of $3$-cocycles.
\end{lemma}

Let $k\Z_l$ be the group algebra of the cyclic group $\Z_l=\langle g \rangle$ and for all
$0\leq i\leq l-1$, define
$$1_{i}^{l}:=\frac{1}{l}\sum_{j=0}^{l-1}(\zeta_{l}^{l-i})^{j}g^{j}.$$
It is known that $g1_{i}^{l}=\zeta_{l}^{i}1_{i}^{l}$ and $\{1_{i}^{l}|0\leq i\leq l-1\}$ is a complete set of
primitive idempotents of $k\mathbb{Z}_{l}$. If $l$ is a square integer, say $l=\mathbbm{l}^{2},$ then let $\mathbbm{g}=g^{\mathbbm{l}}$ and define
$$\mathbbm{1}_{i}^{\mathbbm{l}}:=\frac{1}{\mathbbm{l}}\sum_{j=0}^{\mathbbm{l}-1}(\zeta_{\mathbbm{l}}^{\mathbbm{l}-i})^{j}\mathbbm{g}^{j}.$$
The superscripts of $1_{i}^{l}$ and $\mathbbm{1}_{i}^{\mathbbm{l}}$ will be omitted when there is no risk of confusion.
We have the following identity (see also \cite{G}):
\begin{equation} \sum_{j=0}^{\mathbbm{l}-1}1_{\mathbbm{l}j+i}=\mathbbm{1}_{i}.
\end{equation}

Let $G=\Z_{m}\times \Z_n$ with $m|n$ and $h_{1},h_2$ be standard generators. Assume that $m=\mathbbm{m}^2$ and $n=\mathbbm{n}^2$ and let $\mathbbm{G}:=\mathbbm{Z}_{\mathbbm{m}}
\times \mathbbm{Z}_{\mathbbm{n}},\;\mathbbm{h}_{1}:=h_{1}^{\mathbbm{m}},\;\mathbbm{h}_2:=h_{2}^{\mathbbm{n}}$. Regard $k\mathbbm{G}$ as a quasi-Hopf algebra with associator $\Phi,$  then by Lemma \ref{l5.2}, $\Phi$ is of the form \begin{eqnarray}\label{ass}\Phi_{a,b,c}&=&\sum\omega_{a,b,c}(\mathbbm{g}^{i_{1}}\mathbbm{h}^{i_{2}},\mathbbm{g}^{j_{1}}\mathbbm{h}^{j_{2}},\mathbbm{g}^{k_{1}}\mathbbm{h}^{k_{2}})\mathbbm{1}_{i_{1}}\mathbbm{1}_{i_{2}}
\otimes \mathbbm{1}_{j_{1}}\mathbbm{1}_{j_{2}}\otimes\mathbbm{1}_{k_{1}}\mathbbm{1}_{k_{2}}\notag\\
&=&\sum \zeta_{\mathbbm{m}}^{ai_{1}[\frac{j_{1}+k_1}{\mathbbm{m}}]}\zeta_{(\mathbbm{m},\mathbbm{n})}^{bi_{2}[\frac{j_{1}+k_1}{\mathbbm{m}}]}
\zeta_{\mathbbm{n}}^{ci_{2}[\frac{j_{2}+k_{2}}{\mathbbm{n}}]}\mathbbm{1}_{i_{1}}\mathbbm{1}_{i_{2}}
\otimes \mathbbm{1}_{j_{1}}\mathbbm{1}_{j_{2}}\otimes\mathbbm{1}_{k_{1}}\mathbbm{1}_{k_{2}}\\
&=&\sum \zeta_{\mathbbm{m}}^{(ai_{1}+bi_{2})[\frac{j_{1}+k_1}{\mathbbm{m}}]}
\zeta_{\mathbbm{n}}^{ci_{2}[\frac{j_{2}+k_{2}}{\mathbbm{n}}]}\mathbbm{1}_{i_{1}}\mathbbm{1}_{i_{2}}
\otimes \mathbbm{1}_{j_{1}}\mathbbm{1}_{j_{2}}\otimes\mathbbm{1}_{k_{1}}\mathbbm{1}_{k_{2}}.\notag
\end{eqnarray}
for some $0\leq a,b<\mathbbm{m}$ and $0\leq c< \mathbbm{n},$ up to twist equivalence, see \cite{D}.
 For any integer $i\in \mathbb{N}$,
we denote by $i'$ and $i''$ the remainders of the divisions of $i$ by $\mathbbm{m}$
and $\mathbbm{n}$ respectively.
Define  \begin{eqnarray}\label{eq;5.3}J_{a,b,c}&=&\sum_{x_{1},x_{2}=1}^{m}\sum_{y_{1},y_{2}=1}^{n}\zeta_{m}^{ax_{1}(y_{1}-y_{1}')}\zeta_{\mathbbm{m}(\mathbbm{m},\mathbbm{n})}^{bx_{2}(y_{1}-y_{1}')}
\zeta_{n}^{cx_{2}(y_{2}-y_{2}'')}1_{x_{1}}1_{x_{2}}\otimes 1_{y_{1}}1_{y_{2}}\notag\\
&=& \sum_{x_{1},x_{2},y_{1},y_{2}}\zeta_{m}^{(ax_{1}+bx_{2})(y_{1}-y_{1}')}
\zeta_{n}^{cx_{2}(y_{2}-y_{2}'')}1_{x_{1}}1_{x_{2}}\otimes 1_{y_{1}}1_{y_{2}} \ .\end{eqnarray}

A key observation is as follows, which is known for cyclic groups, see \cite[Lemma 3.5]{G}.
\begin{proposition}\label{p5.4} $d(J_{a,b,c})=\Phi_{a,b,c}$.
\end{proposition}
\begin{proof} For brevity, write $J=J_{a,b,c}.$ Then
\begin{eqnarray*}
d(J)&=&(1\otimes J)(id\otimes\Delta)(J)(\Delta\otimes id)(J^{-1})(J^{-1}\otimes 1)\\
&=&\sum_{x_{1},x_{2},y_{1},y_{2}}\zeta_{m}^{(ax_{1}+bx_{2})(y_{1}-y_{1}')}
\zeta_{n}^{cx_{2}(y_{2}-y_{2}'')} 1\otimes 1_{x_{1}}1_{x_{2}}\otimes 1_{y_{1}}1_{y_{2}}\\
&&\sum_{i_{1},i_{2},j_{11},j_{21},j_{12},j_{22}}\zeta_{m}^{(ai_{1}+bi_{2})(j_{11}+j_{12}-(j_{11}+j_{12})')}
\zeta_{n}^{ci_{2}(j_{21}+j_{22}-(j_{21}+j_{22})'')}\\
&& 1_{i_{1}}1_{i_{2}}\otimes 1_{j_{11}}1_{j_{21}}\otimes 1_{j_{12}}1_{j_{22}}\\
&&\sum_{i_{11},i_{12},i_{21},i_{22},j_{1},j_{2}}\zeta_{m}^{-[a(i_{11}+i_{12})+b(i_{21}+i_{22})](j_{1}-j_{1}')}
\zeta_{n}^{-c(i_{21}+i_{22})(j_{2}-j_{2}'')}\\
&& 1_{i_{11}}1_{i_{21}}\otimes 1_{i_{12}}1_{i_{22}}\otimes 1_{j_{1}}1_{j_{2}}\\
&&\sum_{u_{1},u_{2},v_{1},v_{2}}\zeta_{m}^{-(au_{1}+bu_{2})(v_{1}-v_{1}')}
\zeta_{n}^{-cu_{2}(v_{2}-v_{2}'')}  1_{u_{1}}1_{u_{2}}\otimes 1_{v_{1}}1_{v_{2}}\otimes 1\\
&=&\sum_{i_{1},i_{2},j_{1},j_{2},k_{1},k_{2}}\zeta_{m}^{(ai_{1}+bi_{2})(j_{1}'+k_{1}'-(j_{1}+k_{1})')}
\zeta_{n}^{ci_{2}(j_{2}''+k_{2}''-(j_{2}+k_{2})'')}\\
&&1_{i_{1}}1_{i_{2}}\otimes 1_{j_{1}}1_{j_{2}}\otimes 1_{k_{1}}1_{k_{2}}\\
&=&\sum_{i_{1},i_{2},j_{1},j_{2},k_{1},k_{2}}\zeta_{\mathbbm{m}}^{(ai_{1}+bi_{2})[\frac{j_{1}+k_1}{m}]}
\zeta_{\mathbbm{n}}^{ci_{2}[\frac{j_{2}+k_{2}}{n}]}
\mathbbm{1}_{i_{1}}\mathbbm{1}_{i_{2}}
\otimes \mathbbm{1}_{j_{1}}\mathbbm{1}_{j_{2}}\otimes\mathbbm{1}_{k_{1}}\mathbbm{1}_{k_{2}}\\
&=&\Phi_{a,b,c}.
\end{eqnarray*}
\end{proof}

\subsection{Construction.} Let $H:=H(m,n,l_{1},l_{2},g,h)$ and set $X:=g^{-\frac{\ord(g)}{2}}h^{-l_{2}}x,\;Y:=g^{-l_{1}}h^{-\frac{\ord(h)}{2}}y$.
Assume that $m=\mathbbm{m}^2$ and $n=\mathbbm{n}^2$.
By Proposition \ref{p1}, $g=g_{2}h_{1}h_{2}^{a}$ and
$h=g_{1}g_{2}^{b}h_{2}.$ We keep the notations of Section 4. Note that $m_{i}=\mathbbm{m}_{i}^{2}$ and
$n_{i}=\mathbbm{n}_{i}^{2}$ for $i=1,2$. As in the above subsection, let $J=J_{a,b,c}$ for some $0\leq a,b<\mathbbm{m}$ and $0\leq c< \mathbbm{n}$. Consider the subalgebra $A(H, J)\subset H$ which is generated by $X,Y$ and $\mathbbm{g}_{i}:=g_{i}^{\mathbbm{m}_{i}},\; \mathbbm{h}_{i}:=h_{i}^{\mathbbm{n}_{i}}$ for $i=1,2$. The main task of this subsection is to determine when $A(H,J)$ is a quasi-Hopf subalgebra of $H^{J}$.\\[2mm]
\textbf{Convention:} For convenience, $A(H,J)$ is abbreviated as $A(H)$ when it is clear from the context.

We start with a special case in which the order of one of $g,h$, say $h$, is $n$. By Corollary \ref{c2}, we can take $g=h_{1}h_{2}^{\sigma}$ and $h=h_{2}$ for some $0\leq \sigma < n$, where $h_{1},h_{2}$ are
standard generators of $\mathbbm{Z}_{m}\times\mathbbm{Z}_{n}$. It turns out that this case already sheds much light on the general situation.

\begin{lemma}\label{l5.5} Assume that $g=h_{1}h_{2}^{\sigma}$ and $h=h_{2}$ for some $0\leq \sigma < n$. For $H=H(m,n,l_{1},l_{2},g,h)$,  we have

\emph{(i)} $\ord(g)=m,\;\;n|\sigma m$,

\emph{(ii)} $h_{1}X=\zeta^{-1}_{m} Xh_{1},\;h_{1}Y=\zeta_{m}^{\frac{\sigma m}{n}} Yh_{1},\;h_{2}X= Xh_{2},\;h_{2}Y=\zeta^{-1}_{n} Yh_{2}$,

\emph{(iii)} $1^{m}_{i}X=X1^{m}_{i+1},\;1^{m}_{i}Y=Y1^{m}_{i-\frac{\sigma m}{n}}$,

\emph{(iv)} $1^{n}_{i}X=X 1^{n}_{i},\;1^{n}_{i}Y=Y1^{n}_{i+1}$.
\end{lemma}
\begin{proof} Since $h_{1}$ is generated by $g$ and $h$, there is $\xi\in k^{\ast}$ such that $h_{1}Xh_{1}^{-1}=\xi X$. By the definition of $H$,
$gXg^{-1}=\zeta_{\ord(g)}^{-1}X.$ At the same time, $$gXg^{-1}=h_{1}h_{2}^{\sigma}X(h_{1}h_{2}^{\sigma})^{-1}=h_{1}Xh_{1}^{-1}=\xi X.$$
Therefore, $\xi=\zeta^{-1}_{\ord(g)}$. So $\ord(g)\leq \ord(h_{1})$. On the other hand $g=h_{1}h_{2}^{\sigma}$,  and so $\ord(g) \geq \ord(h_{1})$. Thus $\ord(g)=\ord(h_{1})=m.$ This implies that $1=g^{m}=(h_{1}h_{2}^{\sigma})^{m}=(h_{2}^{\sigma})^{m}$. So $n|\sigma m$. We proved (i).

For (ii), the first equality already appeared in the proof of (i). We only need to prove the second equality as the last two equalities are just the definition of $H$. In fact, by definition, $gY=Yg$. That is, $h_{1}h_{2}^{\sigma}Y=Yh_{1}h_{2}^{\sigma}$. Since $h_{2}Y=\zeta^{-1}_{n} Yh_{2}$, $h_{1}h_{2}^{\sigma}Y=\zeta_{n}^{-\sigma}h_{1}Yh_{2}^{\sigma}$.
Therefore, $\zeta_{n}^{-\sigma}h_{1}Y=Yh_{1}$ which implies that $h_{1}Y=\zeta_{n}^{\sigma}Yh_{1}=\zeta_{m}^{\frac{\sigma m}{n}} Yh_{1}$.

For (iii), we have
\begin{eqnarray*}1^{m}_{i}X&=&\frac{1}{m}\sum_{j=0}^{m-1}\zeta_{m}^{-ij}h_{1}^{j}X= \frac{1}{m}\sum_{j=0}^{m-1}\zeta_{m}^{-ij}\zeta_{m}^{-j}X h_{1}^{j}\\
&=&X(\frac{1}{m}\sum_{j=0}^{m-1}\zeta_{m}^{-(i+1)j} h_{1}^{j})=X1^{m}_{i+1}.
\end{eqnarray*}
\begin{eqnarray*}1^{m}_{i}Y&=&\frac{1}{m}\sum_{j=0}^{m-1}\zeta_{m}^{-ij}h_{1}^{j}Y= \frac{1}{m}\sum_{j=0}^{m-1}\zeta_{m}^{-ij}\zeta_{m}^{\frac{\sigma mj}{n}}Y h_{1}^{j}\\
&=&Y(\frac{1}{m}\sum_{j=0}^{m-1}\zeta_{m}^{-(i-\frac{\sigma m}{n})j} h_{1}^{j})=Y1^{m}_{i-\frac{\sigma m}{n}}.
\end{eqnarray*}
 The proof of (iv) is similar to (iii) and so we omit it.
\end{proof}

\begin{proposition}\label{p5.6} Assume that $g=h_{1}h_{2}^{\sigma}$ and $h=h_{2}$ for some $0\leq \sigma < n$. Then $A(H)$ is a quasi-Hopf subalgebra of $H^{J}$ if and only if
$$a=0,\;\;\mathbbm{n}|l_{2},\;\;l_{1}+b\equiv 0(\mathbbm{m}),\;\;c+\sigma l_{1}\equiv 0(\mathbbm{n}).$$
\end{proposition}
\begin{proof} By Proposition \ref{p5.4}, we know $\Phi_{a,b,c}\in A(H)\otimes A(H)\otimes A(H)$. Again, write $J=J_{a,b,c}$. In the proof of the following equations, identities in Lemma \ref{l5.5} are used freely.

\begin{eqnarray*}\Delta_{J}(X)&=&J\Delta(X)J^{-1}=\sum_{x_{i},y_{i}}\zeta_{m}^{(ax_{1}+bx_{2})
(y_{1}-y'_{1})}\zeta_{n}^{cx_{2}(y_{2}-y''_{2})} 1_{x_{1}}1_{x_{2}}\otimes 1_{y_{1}}1_{y_{2}}\\
&& \times (X\otimes g^{-\frac{\ord(g)}{2}}h^{-l_{2}}+1\otimes X)\\
&& \times \sum_{s_{i},t_{i}}\zeta_{m}^{-(as_{1}+bs_{2})
(t_{1}-t'_{1})}\zeta_{n}^{-cs_{2}(t_{2}-t''_{2})} 1_{s_{1}}1_{s_{2}}\otimes 1_{t_{1}}1_{t_{2}}\\
&=&\sum_{x_{i},y_{i},s_{i},t_{i}}\zeta_{m}^{(ax_{1}+bx_{2})
(y_{1}-y'_{1})}\zeta_{n}^{cx_{2}(y_{2}-y''_{2})}\zeta_{m}^{-(as_{1}+bs_{2})
(t_{1}-t'_{1})}\zeta_{n}^{-cs_{2}(t_{2}-t''_{2})} \\
&&\zeta_{m}^{-\frac{my_{1}}{2}}\zeta_{n}^{-\frac{m\sigma y_{2}}{2}-l_{2}y_{2}} X1_{x_{1}+1}1_{x_{2}}1_{s_{1}}1_{s_{2}}\otimes 1_{y_{1}}1_{y_{2}}1_{t_{1}}1_{t_{2}}\\
&& +\sum_{x_{i},y_{i},s_{i},t_{i}}\zeta_{m}^{(ax_{1}+bx_{2})
(y_{1}-y'_{1})}\zeta_{n}^{cx_{2}(y_{2}-y''_{2})}\zeta_{m}^{-(as_{1}+bs_{2})
(t_{1}-t'_{1})}\zeta_{n}^{-cs_{2}(t_{2}-t''_{2})} \\
&&1_{x_{1}}1_{x_{2}}1_{s_{1}}1_{s_{2}}\otimes X1_{y_{1}+1}1_{y_{2}}1_{t_{1}}1_{t_{2}}\\
&=&\sum_{x_{i},y_{i}}\zeta_{m}^{(ax_{1}+bx_{2})
(y_{1}-y'_{1})}\zeta_{n}^{cx_{2}(y_{2}-y''_{2})}\zeta_{m}^{-(ax_{1}+a+bx_{2})
(y_{1}-y'_{1})}\zeta_{n}^{-cx_{2}(y_{2}-y''_{2})} \\
&&\zeta_{m}^{-\frac{my_{1}}{2}}\zeta_{n}^{-\frac{m\sigma y_{2}}{2}-l_{2}y_{2}} X1_{x_{1}+1}1_{x_{2}}\otimes 1_{y_{1}}1_{y_{2}}\\
&&+ \sum_{x_{i},y_{i}}\zeta_{m}^{(ax_{1}+bx_{2})
(y_{1}-y'_{1})}\zeta_{n}^{cx_{2}(y_{2}-y''_{2})}\zeta_{m}^{-(ax_{1}+bx_{2})
(y_{1}+1-(y_{1}+1)')}\zeta_{n}^{-cx_{2}(y_{2}-y''_{2})} \\
&&1_{x_{1}}1_{x_{2}}\otimes X1_{y_{1}+1}1_{y_{2}}\\
&=&\sum_{x_{i},y_{i}}\zeta_{m}^{-a(y_{1}-y'_{1})-\frac{my_{1}}{2}}\zeta_{n}^{-\frac{m\sigma y_{2}}{2}-l_{2}y_{2}} X1_{x_{1}+1}1_{x_{2}}\otimes 1_{y_{1}}1_{y_{2}}\\
&&+ \sum_{x_{i},y_{i}}\zeta_{m}^{-(ax_{1}+bx_{2})
(y'_{1}+1-(y_{1}+1)')} 1_{x_{1}}1_{x_{2}}\otimes X1_{y_{1}+1}1_{y_{2}}.
\end{eqnarray*}
Consider the first item of this expression, clearly $$\sum_{x_{i},y_{i}}\zeta_{m}^{-a(y_{1}-y'_{1})-\frac{my_{1}}{2}}\zeta_{n}^{-\frac{m\sigma y_{2}}{2}-l_{2}y_{2}} X1_{x_{1}+1}1_{x_{2}}\otimes 1_{y_{1}}1_{y_{2}}\in A(H)\otimes A(H)$$ if and only if $\mathbbm{m}|(a+\frac{m}{2})$ and $\mathbbm{n}|(l_{2}+\frac{m\sigma}{2})$.
Since we already have $\mathbbm{m}|\frac{m}{2}$ ($m$ is even, by assumption) and $\mathbbm{n}|\frac{m\sigma}{2}$,
the condition is equivalent to $\mathbbm{m}|a$ and $\mathbbm{n}|l_{2}$. So $a=0$. It is not hard to see that the second term always belongs to $A(H)\otimes A(H)$.
Therefore,  $\Delta_{J}(X)\in A(H)\otimes A(H)$ if and only if $a=0,\;\mathbbm{n}|l_{2}$.

Now consider $\Delta_{J}(Y)$.
\begin{eqnarray*}\Delta_{J}(Y)&=&J\Delta(Y)J^{-1}=\sum_{x_{i},y_{i}}\zeta_{m}^{(ax_{1}+bx_{2})
(y_{1}-y'_{1})}\zeta_{n}^{cx_{2}(y_{2}-y''_{2})} 1_{x_{1}}1_{x_{2}}\otimes 1_{y_{1}}1_{y_{2}}\\
&& \times (Y\otimes g^{-l_{1}}h^{-\frac{n}{2}}+1\otimes Y)\\
&& \times \sum_{s_{i},t_{i}}\zeta_{m}^{-(as_{1}+bs_{2})
(t_{1}-t'_{1})}\zeta_{n}^{-cs_{2}(t_{2}-t''_{2})} 1_{s_{1}}1_{s_{2}}\otimes 1_{t_{1}}1_{t_{2}}\\
&=&\sum_{x_{i},y_{i},s_{i},t_{i}}\zeta_{m}^{(ax_{1}+bx_{2})
(y_{1}-y'_{1})}\zeta_{n}^{cx_{2}(y_{2}-y''_{2})}\zeta_{m}^{-(as_{1}+bs_{2})
(t_{1}-t'_{1})}\zeta_{n}^{-cs_{2}(t_{2}-t''_{2})} \\
&&\zeta_{m}^{-l_{1}y_{1}}\zeta_{n}^{-\sigma l_{1} y_{2}-\frac{ny_{2}}{2}} Y1_{x_{1}-\frac{\sigma m}{n}}1_{x_{2}+1}1_{s_{1}}1_{s_{2}}\otimes 1_{y_{1}}1_{y_{2}}1_{t_{1}}1_{t_{2}}\\
&& +\sum_{x_{i},y_{i},s_{i},t_{i}}\zeta_{m}^{(ax_{1}+bx_{2})
(y_{1}-y'_{1})}\zeta_{n}^{cx_{2}(y_{2}-y''_{2})}\zeta_{m}^{-(as_{1}+bs_{2})
(t_{1}-t'_{1})}\zeta_{n}^{-cs_{2}(t_{2}-t''_{2})} \\
&&1_{x_{1}}1_{x_{2}}1_{s_{1}}1_{s_{2}}\otimes Y1_{y_{1}-\frac{\sigma m}{n}}1_{y_{2}+1}1_{t_{1}}1_{t_{2}}\\
&=&\sum_{x_{i},y_{i}}\zeta_{m}^{(ax_{1}+bx_{2})
(y_{1}-y'_{1})}\zeta_{n}^{cx_{2}(y_{2}-y''_{2})}\zeta_{m}^{-(a(x_{1}-\frac{\sigma m}{n})+b(x_{2}+1))
(y_{1}-y'_{1})}\zeta_{n}^{-c(x_{2}+1)(y_{2}-y''_{2})} \\
&&\zeta_{m}^{-l_{1}y_{1}}\zeta_{n}^{-\sigma l_{1} y_{2}-\frac{ny_{2}}{2}} Y1_{x_{1}-\frac{\sigma m}{n}}1_{x_{2}+1}\otimes 1_{y_{1}}1_{y_{2}}\\
&&+ \sum_{x_{i},y_{i}}\zeta_{m}^{(ax_{1}+bx_{2})
(y_{1}-y'_{1})}\zeta_{n}^{cx_{2}(y_{2}-y''_{2})}\zeta_{m}^{-(ax_{1}+bx_{2})
(y_{1}-\frac{\sigma m}{n}-(y_{1}-\frac{\sigma m}{n})')}\\
&&\zeta_{n}^{-cx_{2}(y_{2}+1-(y_{2}+1)'')}
1_{x_{1}}1_{x_{2}}\otimes Y1_{y_{1}-\frac{\sigma m}{n}}1_{y_{2}+1}\\
&=&\sum_{x_{i},y_{i}}\zeta_{m}^{-(-a\frac{\sigma m}{n}+b)
(y_{1}-y'_{1})-l_{1}y_{1}}\zeta_{n}^{-c(y_{2}-y''_{2})-\sigma l_{1} y_{2}-\frac{ny_{2}}{2}} Y1_{x_{1}-\frac{\sigma m}{n}}1_{x_{2}+1}\otimes 1_{y_{1}}1_{y_{2}}\\
&&+ \sum_{x_{i},y_{i}}\zeta_{m}^{(ax_{1}+bx_{2})
((y_{1}-\frac{\sigma m}{n})'-y'_{1}+\frac{\sigma m}{n})}\zeta_{n}^{cx_{2}((y_{2}+1)''-y''_{2}-1)}
1_{x_{1}}1_{x_{2}}\otimes Y1_{y_{1}-\frac{\sigma m}{n}}1_{y_{2}+1}.
\end{eqnarray*}
Similar to the analysis for $\Delta_{J}(X)$, it follows that $\Delta_{J}(Y)\in A(H)\otimes A(H)$ if and only if $l_{1}+b\equiv 0(\mathbbm{m}),\;c+\sigma l_{1}\equiv 0(\mathbbm{n})$.

Since $$\alpha_{J}=\sum_{x_{1},x_{2}}\zeta_{m}^{(ax_{1}+bx_{2})(x_{1}-x'_{1})}\zeta_{n}^{cx_{2}(x_{2}-x_{2}'')}1_{x_{1}}1_{x_{2}},$$
$$\beta_{J}=\sum_{x_{1},x_{2}}\zeta_{m}^{(ax_{1}+bx_{2})((m-x_{1})-(m-x_{1})')}\zeta_{n}^{cx_{2}((n-x_{2})-(n-x_{2})'')}1_{x_{1}}1_{x_{2}},$$
so $$\alpha_{J}\beta_{J}=\sum_{x_{1},x_{2}}\zeta_{m}^{-(ax_{1}+bx_{2})((m-x_{1})'+x_{1}')}\zeta_{n}^{-cx_{2}((n-x_{2})''+x_{2}'')}1_{x_{1}}1_{x_{2}}..$$
Clearly, the coefficients of $1_{s\mathbbm{m}+x_{1}}$ (resp. $1_{t\mathbbm{n}+x_{2}}$) and $1_{x_{1}}$ (resp. $1_{x_{2}}$) are identical, thus $\alpha_{J}\beta_{J}\in A(H)$.

Under the assumption that $a=0$, we have
\begin{eqnarray*}S_{J}(X)&=&\beta_{J}S(X)\beta_{J}^{-1}=
\sum_{x_{1},x_{2}}\zeta_{m}^{bx_{2}((m-x_{1})-(m-x_{1})')}\zeta_{n}^{cx_{2}((n-x_{2})-(n-x_{2})'')}1_{x_{1}}1_{x_{2}}\\
&& \times (-Xg^{\frac{m}{2}}h^{l_{2}})\sum_{s_{1},s_{2}}\zeta_{m}^{-bs_{2}((m-s_{1})-(m-s_{1})')}\zeta_{n}^{-cs_{2}((n-s_{2})-(n-s_{2})'')}1_{s_{1}}1_{s_{2}}\\
&=& -\sum_{x_{1},x_{2}}\zeta_{m}^{bx_{2}((m-x_{1})-(m-x_{1})')}\zeta_{n}^{cx_{2}((n-x_{2})-(n-x_{2})'')}X1_{x_{1}+1}1_{x_{2}}\\
&&\times \zeta_{m}^{\frac{m}{2}(x_{1}+1)}\zeta_{n}^{\frac{m\sigma x_{2}}{2}+l_{2}x_{2}}\sum_{s_{1},s_{2}}\zeta_{m}^{-bs_{2}((m-s_{1})-(m-s_{1})')}\zeta_{n}^{-cs_{2}((n-s_{2})-(n-s_{2})'')}1_{s_{1}}1_{s_{2}}\\
&=&-\sum_{x_{1},x_{2}}\zeta_{m}^{bx_{2}((m-x_{1}-1)'-(m-x_{1})'+1)+\frac{m}{2}(x_{1}+1)}\zeta_{n}^{\frac{m\sigma x_{2}}{2}+l_{2}x_{2}}X1_{x_{1}+1}1_{x_{2}}.
\end{eqnarray*}
From this, we find that $S_{J}(X)\in A(H)$ if $\mathbbm{n}|l_{2}$.
Also, assume that $a=0$, we have
\begin{eqnarray*}S_{J}(Y)&=&\beta_{J}S(Y)\beta_{J}^{-1}=
\sum_{x_{1},x_{2}}\zeta_{m}^{bx_{2}((m-x_{1})-(m-x_{1})')}\zeta_{n}^{cx_{2}((n-x_{2})-(n-x_{2})'')}1_{x_{1}}1_{x_{2}}\\
&& \times (-Xg^{l_{1}}h^{\frac{n}{2}})\sum_{s_{1},s_{2}}\zeta_{m}^{-bs_{2}((m-s_{1})-(m-s_{1})')}\zeta_{n}^{-cs_{2}((n-s_{2})-(n-s_{2})'')}1_{s_{1}}1_{s_{2}}\\
&=& -\sum_{x_{1},x_{2}}\zeta_{m}^{bx_{2}((m-x_{1})-(m-x_{1})')}\zeta_{n}^{cx_{2}((n-x_{2})-(n-x_{2})'')}Y1_{x_{1}-\frac{\sigma m}{n}}1_{x_{2}+1}\\
&&\times \zeta_{m}^{l_{1}(x_{1}-\frac{\sigma m}{n})}\zeta_{n}^{\sigma l_{1}(x_{1}+1)+\frac{n}{2}(x_{2}+1)}\\
&&\sum_{s_{1},s_{2}}\zeta_{m}^{-bs_{2}((m-s_{1})-(m-s_{1})')}\zeta_{n}^{-cs_{2}((n-s_{2})-(n-s_{2})'')}1_{s_{1}}1_{s_{2}}\\
&=&-\sum_{x_{1},x_{2}}\zeta_{m}^{bx_{2}(-\frac{\sigma m}{n}-(m-x_{1})'+(m-x_{1}+\frac{\sigma m}{n})')+
   (l_{1}+b)(x_{1}-\frac{\sigma m}{n})+b(m-x_{1}+\frac{\sigma m}{n})'}\\
&& \zeta_{n}^{cx_{2}(1-(n-x_{2})''-(n-x_{2}-1)'')+(c+\sigma l_{1}+\frac{n}{2})(x_{2}+1)+c(n-x_{2}-1)''}Y1_{x_{1}-\frac{\sigma m}{n}}1_{x_{2}+1}.
\end{eqnarray*}

Similarly, one can show that
$S_{J}(Y)\in A(H)$ if $a=0, l_{1}+b\equiv 0(\mathbbm{m})$ and $c+\sigma l_{1}\equiv 0(\mathbbm{n})$.
\end{proof}

A quasi-Hopf algebra is said to be \emph{genuine} if it is not twist equivalent to a Hopf algebra.

\begin{proposition}\label{p5.7} If $b\neq 0$ or $c\neq 0$, then the quasi-Hopf algebra $(A(H),\Phi_{0,b,c})$ constructed in the above proposition is genuine.
\end{proposition}
\begin{proof} Otherwise, $A(H)^{J}$ is a Hopf algebra for some invertible element $J\in A(H)\otimes A(H)$. Then consider the its
degree zero component $$J_{0}=\sum_{x_i,y_i}\delta(x_{1},x_{2},y_{1},y_{2})\mathbbm{1}_{x_{1}}\mathbbm{1}_{x_{2}}\otimes \mathbbm{1}_{y_{1}}\mathbbm{1}_{y_{2}}.$$
Therefore, $\Phi_{0,b,c}=dJ_{0}$, which is equivalent to saying $\omega_{0,b,c}=d \delta,$ a contradiction.
\end{proof}

Finally we are ready to consider the general case, that is, without the assumption that $h$ is standard.
According to Proposition \ref{p1}, there are integers $m_{1},m_{2},n_{1},n_{2},a,b$ such that $\mathbbm{Z}_{m}\times \mathbbm{Z}_{n}=
(\mathbbm{Z}_{m_{1}}\times \mathbbm{Z}_{m_{2}})\times (\mathbbm{Z}_{n_{1}}\times \mathbbm{Z}_{n_{2}})$
with $m_{1}|m_{2},n_{1}|n_{2},\;(m_{2},n_{2})=1$ and $g=g_{2}h_{1}h_{2}^{a},\;h=g_{1}g_{2}^{b}h_{2}$. Here $g_1,g_{2}$ (resp. $h_{1},h_{2}$) are standard generators of $\mathbbm{Z}_{m_1}\times \mathbbm{Z}_{m_2}$ (resp. $\mathbbm{Z}_{n_1}\times \mathbbm{Z}_{n_2}$).

Assume that $m_{i}=\mathbbm{m}^{2}_{i}$ and $n_{i}=\mathbbm{n}^{2}_{i}$.. By Lemma \ref{l5.2}, every
associator of the group algebra $k \mathbbm{Z}_{m}\times \mathbbm{Z}_{n}$ (viewing as a quasi-Hopf algebra) is of the following form
\begin{eqnarray*}\Phi_{a_1,b_1,c_1}\Phi_{a_2,b_2,c_2}&=&\sum \zeta_{\mathbbm{m}_{1}}^{(a_{1}i_{1}+b_{1}i_{2})[\frac{j_{1}+k_1}{\mathbbm{m}_{1}}]}
\zeta_{\mathbbm{m}_{2}}^{c_{1}i_{2}[\frac{j_{2}+k_{2}}{\mathbbm{m}_{2}}]}\mathbbm{1}^{\mathbbm{m}_{1}}_{i_{1}}\mathbbm{1}^{\mathbbm{m}_{2}}_{i_{2}}
\otimes \mathbbm{1}^{\mathbbm{m}_{1}}_{j_{1}}\mathbbm{1}^{\mathbbm{m}_{2}}_{j_{2}}\otimes\mathbbm{1}^{\mathbbm{m}_{1}}_{k_{1}}\mathbbm{1}^{\mathbbm{m}_{2}}_{k_{2}}\\
&&\times \sum \zeta_{\mathbbm{n}_{1}}^{(a_{2}x_{1}+b_{2}x_{2})[\frac{y_{1}+z_1}{\mathbbm{n}_{1}}]}
\zeta_{\mathbbm{n}_{2}}^{c_{2}x_{2}[\frac{y_{2}+z_{2}}{\mathbbm{n}_{2}}]}\mathbbm{1}^{\mathbbm{n}_{1}}_{x_{1}}\mathbbm{1}^{\mathbbm{n}_{2}}_{x_{2}}
\otimes \mathbbm{1}^{\mathbbm{n}_{1}}_{y_{1}}\mathbbm{1}^{\mathbbm{n}_{2}}_{y_{2}}\otimes\mathbbm{1}^{\mathbbm{n}_{1}}_{z_{1}}\mathbbm{1}^{\mathbbm{n}_{2}}_{z_{2}}.
\end{eqnarray*}
for some $0\leq a_{1},b_{1}<\mathbbm{m_{1}}$ (resp. $0\leq a_{2},b_{2}<\mathbbm{n_{1}}$) and $0\leq c_1< \mathbbm{m}_{2}$ (resp. $0\leq c_2< \mathbbm{n}_{2}$ ), up to twist equivalence.
Similar to Proposition \ref{p5.4}, we can construct $J_{a_{1},b_{1},c_{1}}$ (resp. $J_{a_{2},b_{2},c_{2}}$) and show that
$d(J_{a_{1},b_{1},c_{1}})=\Phi_{a_1,b_1,c_1}$ and $d(J_{a_{2},b_{2},c_{2}})= \Phi_{a_2,b_2,c_2}$.

So, in one word, the group $\mathbbm{Z}_{m}\times \mathbbm{Z}_{n}$ ``splits" into two parts and every part is the same as the special case considered before. Therefore, we state
the following conclusion without proof.

\begin{theorem}\label{t5.8} Using notions given above and setting $J=J_{a_{1},b_{1},c_{1}}J_{a_{2},b_{2},c_{2}}$, we have

\emph{(i)} $A(H)$ is a quasi-Hopf subalgebra of $H^{J}$ if and only if
$$a_1=0,\;\;\mathbbm{m}_{2}|l_{1},\;\;l_{2}+b_{1}\equiv 0(\mathbbm{m}_{1}),\;\;c_{1}+b l_{2}\equiv 0(\mathbbm{m}_{2}),$$
$$a_{2}=0,\;\;\mathbbm{n}_{2}|l_{2},\;\;l_{1}+b_{2}\equiv 0(\mathbbm{n}_{1}),\;\;c_{2}+a l_{1}\equiv 0(\mathbbm{n}_{2}).$$

\emph{(ii)} If $(b_{1},b_{2},c_{1},c_{2})\neq (0,0,0,0)$, then $A(H)$ is genuine.
\end{theorem}

\section{Graded elementary quasi-Hopf algebras of tame type}

In this section, the structures of tame (radically) graded elementary quasi-Hopf algebras will be determined. We begin with a basic observation.

\begin{proposition}\label{p6.1} The quasi-Hopf algebras $A(H)$ are tame.
\end{proposition}
\begin{proof} Let $H:=H(m,n,l_{1},l_{2},g,h)$. For convenience, let $m_{1}=m$ and $m_{2}=n$. Assume that $m_{i}=\mathbbm{m}_{i}^{2}$ for $i=1,2$. Let
 $h_{1}, h_{2}$ be standard generators of $\mathbbm{Z}_{m_{1}}\times \mathbbm{Z}_{m_{2}}\subset H$. Write $G:=\mathbbm{Z}_{\mathbbm{m}_{1}}\times \mathbbm{Z}_{\mathbbm{m}_{2}}$ and $\mathcal{C}:=\Rep(A(H)).$ We will show that $\Rep(H)$ is equivalent to $\mathcal{C}^{G}$.

For $j_{i}=0,\ldots,\mathbbm{m}_{i}-1$ and $i=1,2$, let $F_{i,j_{i}}:\Rep(A(H))\to \Rep(A(H))$ be the functor defined as follows. For an object $(V,\pi_{V})\in \Rep(A(H))$,
$F_{i,j_{i}}(V)=V$ as vector space, and $\pi_{F_{i,j_{i}}}(a)=\pi_{V}(h_{i}^{j_{i}}ah_{i}^{-j_{i}})$, $a\in A(H)$.

The isomorphism $\gamma_{1j_{1},1k_{1}}:F_{1,j_{1}}(F_{1,k_{1}}(V))\to F_{1,(j_{1}+k_{1})'}(V)$ (resp. $\gamma_{2j_{2},2k_{2}}:F_{2,j_{2}}(F_{2,k_{2}}(V))\to F_{2,(j_{2}+k_{2})''}(V)$) is given by the action
$$(h_{1}^{\mathbbm{m}_{1}})^{\frac{(j_{1}+k_{1})'-j_{1}-k_{1}}{\mathbbm{m}_{1}}}\in A(H)\;\;\;\;(\textrm{resp.} \;(h_{2}^{\mathbbm{m}_{2}})^{\frac{(j_{2}+k_{2})''-j_{2}-k_{2}}{\mathbbm{m}_{2}}}\in A(H))$$\
and $\gamma_{i_{1}j,i_{2}k}=1$ for $i_{1}\neq i_{2}$.

By the definition of $\mathcal{C}^{G}$, an object in it is a representation $V$ of $A(H)$ together with a collection of linear isomorphisms $p_{i,j_{i}}:V\to V$, $j_{i}=0,\ldots,\mathbbm{m}_{i}-1$, $i=1,2$, such that
$$p_{i,j_{i}}(av)=h_{i}^{j_{i}}ah_{i}^{-j_{i}}p_{i,j_{i}}(v),\;\;\forall \ a\in A(H),\;v\in V,$$
and $$p_{1,j_{1}}p_{1,k_{1}}=p_{1,(j_{1}+k_{1})'}(h_{1}^{\mathbbm{m}_{1}})^{\frac{-(j_{1}+k_{1})'+j_{1}+k_{1}}{\mathbbm{m}_{1}}},\;\;
p_{2,j_{2}}p_{2,k_{2}}=p_{2,(j_{2}+k_{2})''}(h_{2}^{\mathbbm{m}_{2}})^{\frac{-(j_{2}+k_{2})''+j_{2}+k_{2}}{\mathbbm{m}_{2}}}.$$
It is now straightforward to verify that this is the same as a representation of $H$, because $H$ is generated by $A(H)$ and the $p_{i,j_{i}}=h_{i}^{j_{i}}$ with exactly the same relations.

Therefore, $\Rep(A(H))^{G}$ is equivalent to $\Rep(H)$ and thus $A(H)$ is tame by Theorem \ref{t3.6}.
\end{proof}

\begin{remark} \emph{(1)} The method proving that $\mathcal{C}^{G}$ is equivalent to $\Rep(H)$  already appeared in the proof of \cite[Theorem 4.2]{EG2} by Etingof and Gelaki.

\emph{(2)} To show Proposition 6.1, one may take a more direct way. It is not hard to see that $A(H)=A\# k(\mathbbm{Z}_{m}\times \mathbbm{Z}_{n})$ where
$A=k\langle x,y\rangle/(x^{2},y^{2},(xy)^{m}-c(yx)^{m})$ for some $0\neq c\in k$. Note that $A$ is a special biserial algebra and thus tame. Therefore $A(H)$ is tame by \cite[Theorem 4.5 ]{L2}. The method adopted here reveals the power of the idea of equivariantization and de-equivariantization.
\end{remark}

\begin{proposition}\label{p6.3} Let $A=\bigoplus_{i\geq 0}A[i]$ be a connected tame graded elementary quasi-Hopf algebra  which is genuine. Then there are
$H=H(m,n,l_{1},l_{2},g,h)$ and $J\in H\otimes H$ such that there exists a graded quasi-Hopf algebra epimorphism $\pi: A\twoheadrightarrow A(H,J)$
which is the identity restricted to degrees $0$ and $1$.
 \end{proposition}
\begin{proof} Using the same argument as for \cite[Proposition 4.3]{HL}, we observe first that $A_{0}$ is the group algebra of an abelian group which is generated by two elements $\mathbbm{g},\mathbbm{h}$.
So $A[0]= k \mathbbm{Z}_{\mathbbm{m}}\times \mathbbm{Z}_{\mathbbm{n}}$ with $\mathbbm{m}|\mathbbm{n}$. Let $\mathbbm{h}_{1},\mathbbm{h}_{2}$ be a tuple of standard generators of $\mathbbm{Z}_{\mathbbm{m}}\times \mathbbm{Z}_{\mathbbm{n}}$. Lemma \ref{l2.4} implies that the representation type number of $A$ is $2$, which is equivalent to the fact
that $A[1]$ is free $A[0]$-module of rank $2.$ 

The following is similar to the proof of \cite[Theorem 3.2.1]{A}. Since $A[1]$ is an $A[0]$-bimodule and $A[0]= k \mathbbm{Z}_{\mathbbm{m}}\times \mathbbm{Z}_{\mathbbm{n}}$, then $A[1]$ has a decomposition
$$A[1]=\bigoplus_{r_{1},r_{2}}A_{r_{1},r_{2}}[1]$$
where $A_{r_{1},r_{2}}[1]=\{x\in A|\mathbbm{h}_{1}x\mathbbm{h}^{-1}_{1}=\zeta_{\mathbbm{m}}^{r_{1}}x,\;\mathbbm{h}_{2}x\mathbbm{h}^{-1}_{2}=
\zeta_{\mathbbm{n}}^{r_{2}}x\}$ for $0\leq r_{1}< \mathbbm{m},\;0\leq r_{1}< \mathbbm{n}$. Let $\hat{A}$ be the tensor algebra of $A[1]$ over $A[0].$ Obviously, it is a quasi-Hopf algebra and we have a canonical surjection $\pi_{1}: \hat{A}\twoheadrightarrow A$. Let $m=\mathbbm{m}^{2}, \;n=\mathbbm{n}^{2},$ and $\chi_{1},\chi_{2}$ be the two automorphisms of $\hat{A}$ defined by
$$\chi_{1}|_{A[0]}=\id,\;\chi_{1}|_{A_{r_{1},r_{2}}[1]}=\zeta_{m}^{r_{1}}\id,\;\;
\chi_{2}|_{A[0]}=\id,\;\chi_{2}|_{A_{r_{1},r_{2}}[1]}=\zeta_{n}^{r_{2}}\id.$$

Let $L$ be the sum of all quasi-Hopf ideals of $\hat{A}$ contained in $\bigoplus_{i\geq 2}\hat{A}[i]$. Then $\ker \pi_{1}\subseteq L$ and $\chi_{i}(L)=L$, so
$\chi_{i}$ acts on $\bar{A}:=\hat{A}/L$ for $i=1,2$. Define $\bar{H}$ to be the quasi-Hopf algebra generated by $\bar{A}$ together with two group-like elements $h_{1},h_{2}$ subject to relations
$$h_{1}^{\mathbbm{m}}=\mathbbm{h}_{1},\;\;h_{2}^{\mathbbm{n}}=\mathbbm{h}_{2},\;\;h_{1}h_{2}=h_{2}h_{1},\;\;h_{i}zh_{i}^{-1}=\chi_{i}(z)$$
for all $z\in \bar{A}$ and $i=1,2$. So $h_{1},h_{2}$ generate a group in $\bar{H}$ which is isomorphic to $\mathbbm{Z}_{m}\times \mathbbm{Z}_{n}$.

By Lemma \ref{l5.2}, we can assume that the associator of $A$ equals to $\Phi_{a,b,c}$ for some $0\leq a,b<\mathbbm{m},\;0\leq c<\mathbbm{n}$. Define $J:=J_{a,b,c}\in \bar{H}\otimes \bar{H}$ just
as equation \eqref{eq;5.3}. By Proposition \ref{p5.4}, $\bar{H}^{J^{-1}}$ is a connected graded Hopf algebra. \\[2mm]
\emph{Claim: $\bar{H}^{J^{-1}}$ is a tame algebra.}\\[1.5mm]
\emph{Proof of this Claim}:
Since $\bar{H}^{J^{-1}}$ is nothing other than $\bar{H}$ as an algebra, it suffices to show that $\bar{H}$ is tame. By the same idea used in the proof 
of Proposition \ref{p6.1}, $\Rep(\bar{H})$ is equivalent to $(\Rep{\bar{A}})^{G}$ for some finite group $G$. Then Theorem \ref{t3.6} implies that $\bar{H}$ and
$\bar{A}$ have the same representation type. So it remains to show that $\bar{A}$ is of tame type. Since $\ker(\pi_{1})\subseteq L$,  there is a quasi-Hopf algebra epimorphism
$\pi:\;A\twoheadrightarrow \bar{A}$. Since $A$ is tame, $\bar{A}$ is of tame type or finite representation type. If $\bar{A}$ is of finite representation type, then
its representation type number is $1$ by Lemma \ref{l2.4}. Since clearly $\pi$ is identity on $A[0]\oplus A[1]$, $A$ and $\bar{A}$ have the same representation type number. Therefore,
$n_{\bar{A}}=2$ and thus $\bar{A}$ is not of finite representation type. This completes the proof of this claim.

By Lemma \ref{l5.1}, $\bar{H}^{J^{-1}}\cong H(m,n,l_{1},l_{2},g,h)$ for some $l_{1}|m,\;l_{2}|n$ and two generators $g,h$ of $\mathbbm{Z}_{m}\times \mathbbm{Z}_{n}$. Regarding this isomorphism as an identity, $\bar{A}=A(H(m,n,l_{1},l_{2},g,h),J)$ and so we have a quasi-Hopf epimorphism $\pi: A\twoheadrightarrow A(H(m,n,l_{1},l_{2},g,h),J)$
which is the identity restricted to degrees $0$ and $1$. Done.
\end{proof}

\begin{proposition} \label{p6.4} The epimorphism $\pi$ given in Proposition \ref{p6.3} is indeed an isomorphism.
\end{proposition}
\begin{proof}  As Theorem \ref{t5.8} indicates that the general situation separates naturally into two independent parts, so  there is no harm to assume that  $A(H,J)$ is just the $A(H)$ considered in Proposition \ref{p5.6} and thus all notations appeared therein are used freely henceforth.  Since $\pi$ is an identity in degree $1$ and degree $0$ parts, we denote the preimage of $X,Y$ under $\pi$ still by $X,Y.$ To prove the assertion, it is enough to show that the relations in $A(H)$
still hold in $A.$ In other words, the equations $X^{2}=0,\; Y^{2}=0$ and $(XY)^{l}+(-q_{2}^{-1})^{l}(YX)^{l}=0$ hold in $A.$ We remark that the last equation appears different from the corresponding relation in $H(m,n,l_{1},l_{2},g,h)$ ($q_2$ is replaced by $q_2^{-1}$), due to the choice of generators $X:=g^{-\frac{\ord(g)}{2}}h^{-l_{2}}x,\;Y:=g^{-l_{1}}h^{-\frac{\ord(h)}{2}}y.$\\[1.5mm]
\emph{Claim: $X^{2}=0$ holds in $A$}. \\[1mm]
\emph{Proof of this claim}:  If not, consider the subalgebra $B\subset A$ generated by $A[0]$ and $X^{2}$.  Of course, $B$ is finite dimensional.
Using the fact that the comultiplication is an algebra map and the formula of $\Delta_{J}$ in $A(H)$,  we have
\begin{eqnarray*}
\Delta_J(X^{2})&=&\sum_{x_{i},y_{i}}\zeta_{m}^{-2a(y_{1}-y'_{1})}\zeta_{n}^{-2l_{2}y_{2}}X^{2}1_{x_{1}+2}1_{x_2}\otimes 1_{y_{1}}1_{y_2}\\
&&+ \sum_{x_i,y_i}\zeta_{m}^{(ax_{1}+bx_{2})((y_{1}+2)'-y_{1}'-2)} 1_{x_{1}}1_{x_2}\otimes X^{2}1_{y_{1}+2}1_{y_2}.
\end{eqnarray*}
So $B$ is indeed a quasi-Hopf subalgebra of $A$.  Similar to the proof of Proposition \ref{p6.3}, let $\hat{B}$ be the tensor algebra $T_{B[0]}B[1].$ Then it is a quasi-Hopf algebra as well. Let $I$ be the unique maximal quasi-Hopf ideal contained in $\bigoplus_{i\geq 2}\hat{B}[i]$ and set $\bar{B}=\hat{B}/I$. In the same manner, we may define $\chi_{1}$ and $\chi_{2}$ as in Proposition \ref{p6.3}. Then we can construct a twist $J$ of $H_{B}:=(\bar{B}\# k[h_{1},h_{2}])/(h_{1}^{\mathbbm{m}}-\mathbbm{h}_{1},h_{2}^{\mathbbm{n}}-\mathbbm{h}_{2})$ such that $H_{B}^{J^{-1}}$ is a Hopf algebra.  In this Hopf algebra, direct computations show that
$\Delta(X^{2})=X^{2}\otimes h_{2}^{-2l_{2}} +1\otimes X^{2},\;\; h_{2}X^{2}=X^{2}h_{2}$. It is well known that this condition will force $(X^{2})^{i}\neq 0$ for all $i\geqslant 1.$ In fact, if otherwise, take $t$ to be smallest number such that $(X^{2})^{t}=0.$ Then $0=\Delta((X^{2})^{t})=\sum_{j=0}^{t}\left ( \begin{array}{c} j\\t
\end{array}\right) (X^{2})^{j}\otimes (X^{2})^{t-j}h^{-2jl_{2}}_{2}$ which leads to the contradiction $(X^{2})^{t-1}=0.$ Therefore, the subalgebra generated by $X^{2}$ is infinite dimensional and thus $B$ is infinite dimensional too. This is a contradiction.

Using the same argument, we have $Y^{2}=0$ in $A$. The proof for the equation $(XY)^{l}+(-q_{2}^{-1})^{l}(YX)^{l}=0$ is more complicated, but the same argument used above still applies. Hence we only sketch the proof and leave the detail to the interested reader. Let $B$ be the subalgebra generated by $A[0]$ and $(XY)^{l}+(-q_{2}^{-1})^{l}(YX)^{l}$. Again, it is a quasi-Hopf subalgebra of $A$ and in $H_{B}^{J^{-1}}$ we have
$$\Delta(Z)=Z\otimes g+1\otimes Z,\;\;gZ=Zg$$
where $Z=(XY)^{l}+(-q_{2}^{-1})^{l}(YX)^{l}$ and
$g=h_{1}^{l(-l_{1}-\frac{m}{2})}h_{2}^{\sigma l(-l_{1}-\frac{m}{2})+l(-l_{2}-\frac{n}{2})},$ see \cite[Lemma 4.14]{HL} for an explanation of the complicated form of $g.$  From this, one may conclude that $H_{B}^{J^{-1}}$ is
infinite dimensional, which contradicts to the fact that $B$ is finite dimensional.
\end{proof}

Combining the preceding two propositions, our desired main result follows.

\begin{theorem} Let $A=\bigoplus_{i\geq 0}A[i]$ be a connected tame graded elementary quasi-Hopf algebra. Then $A$ is twist equivalent to one of the following
quasi-Hopf algebras:
\begin{itemize}
\item[(i)]  $H(m,n,l_{1},l_{2},g,h)$
for some $m,n,l_{1},l_{2}$ and two generators $g,h$ of $\mathbbm{Z}_{m}\times \mathbbm{Z}_{n}$,
\item[(ii)]  $A(H,J)$ for some $H=H(m,n,l_{1},l_{2},g,h)$ and twist $J\in H\otimes H$.
\end{itemize}
\end{theorem}

\end{document}